\documentclass[12pt]{amsart}
\usepackage{}

\usepackage{amsmath}
\usepackage{amsfonts}
\usepackage{amssymb}
\usepackage[all]{xy}           

\usepackage{bbding}
\usepackage{txfonts}
\usepackage{amscd}

\usepackage[shortlabels]{enumitem}
\usepackage{ifpdf}
\ifpdf
  \usepackage[colorlinks,final,backref=page,hyperindex]{hyperref}
\else
  \usepackage[colorlinks,final,backref=page,hyperindex,hypertex]{hyperref}
\fi
\usepackage{tikz}
\usepackage[active]{srcltx}

\makeatletter

\newtheorem{thm}{Theorem}[section]
\newtheorem{lem}[thm]{Lemma}
\newtheorem{cor}[thm]{Corollary}
\newtheorem{pro}[thm]{Proposition}
\newtheorem{ex}[thm]{Example}
\newtheorem{rmk}[thm]{Remark}
\newtheorem{defi}[thm]{Definition}

\newcommand {\emptycomment}[1]{}
\newcommand {\yh}[1]{{\marginpar{*}\scriptsize\textcolor{purple}{yh: #1}}}

\newcommand{\lon }{\,\rightarrow\,}
\newcommand{\be }{\begin{equation}}
\newcommand{\ee }{\end{equation}}

\newcommand{\g}{\mathfrak g}
\newcommand{\h}{\mathfrak h}

\newcommand{\R}{\mathbb R}

\newcommand{\huaB}{\mathcal{B}}

\newcommand{\huaH}{\mathcal{H}}

\newcommand{\huaO}{{\mathcal{O}}}

\newcommand{\frkk}{\mathfrak k}

\newcommand{\dM}{\mathrm{d}}

\newcommand{\RBGH}{\mathsf{RB_G^H}}
\newcommand{\lRBGH}{\mathsf{LRB_G^H}}
\newcommand{\RBgh}{\mathsf{RB_\g^\h}}

\newcommand{\Id}{{\rm{Id}}}

\newcommand{\br}[1]{   [ \cdot,    \cdot  ]   }

\newcommand{\CE}{\mathsf{CE}}
\newcommand{\VE}{\mathsf{VE}}

\newcommand{\Hom}{\mathrm{Hom}}
\newcommand{\Exp}{\mathfrak{ exp}}
\newcommand{\EXP}{\mathrm{Exp}}

\newcommand{\Diff}{\mathbf{\mathfrak{D}}}
\newcommand{\Int}{\mathfrak{I}}
\newcommand{\Der}{\mathrm{Der}}

\newcommand{\Ad}{\mathrm{Ad}}
\newcommand{\Aut}{\mathrm{Aut}}

\newcommand{\gl}{\mathfrak {gl}}

\newcommand{\so}{\mathfrak {so}}

\newcommand{\SO}{\mathrm{SO}}

\newcommand{\ad}{\mathrm{ad}}

\newcommand{\Img}{\mathrm{Im}}

\newcommand{\Sym}{\mathsf{S}}

\newcommand{\tH}{\tilde{H}}

\newcommand{\comment}[1]{{{\color{green}*}\color{blue}{\bf{}*Comment:\\}\scriptsize{\
#1 \ }}}

\begin{document}

\title{Lie  theory and cohomology of relative Rota-Baxter operators}

\author{Jun Jiang}
\address{Department of Mathematics, Jilin University, Changchun 130012, Jilin, China}
\email{jiangjun20@mails.jlu.edu.cn}

\author{Yunhe Sheng}
\address{Department of Mathematics, Jilin University, Changchun 130012, Jilin, China}
\email{shengyh@jlu.edu.cn}

\author{Chenchang Zhu}
\address{Mathematics Institute, Georg-August-University Gottingen, Bunsenstrasse 3-5 37073, Gottingen, Germany}
\email{czhu@gwdg.de}


\begin{abstract}
In this paper, we establish a local Lie theory for relative Rota-Baxter
operators of weight $1$. First we recall the category of relative
Rota-Baxter operators of weight $1$ on Lie algebras and construct a
cohomology theory for them. We use the second cohomology group to
study infinitesimal deformations of relative Rota-Baxter operators and modified
$r$-matrices. Then we introduce a cohomology theory of   relative
Rota-Baxter operators on a Lie group. We construct the differentiation functor from the
category of relative Rota-Baxter operators on Lie groups to that on
Lie algebras, and extend it to
the cohomology level by proving the Van Est theorem
between the two cohomology theories. We integrate a relative
Rota-Baxter operator of weight 1 on a Lie algebra to a local relative Rota-Baxter
operator on the corresponding  Lie group, and show that the local
integration and differentiation are adjoint to each other. Finally, we give two applications of our integration of Rota-Baxter operators: one is to give an explicit formula for the factorization problem, and the other is to provide an integration for matched pairs.
\end{abstract}


\maketitle

\tableofcontents

\allowdisplaybreaks


\section{Introduction}

The notion of Rota-Baxter operators on associative algebras was introduced by G. Baxter and they are applied in the
Connes-Kreimer's algebraic approach to renormalization of quantum
field theory~\cite{CK,Gub}. In the corresponding semi-classical world, the notion of relative Rota-Baxter operators (also called $\huaO$-operators) on Lie algebras was introduced in \cite{Ku}, and they are closely related to various classical Yang-Baxter equations. Moreover, from a more algebraic approach, a relative Rota-Baxter operator naturally gives rise to a pre-Lie or a post-Lie
algebra, and play important roles in mathematical
physics~\cite{BGN,Bu}.

A Rota-Baxter operator of weight 0 on a Lie
algebra is naturally the operator form of a classical skew-symmetric
$r$-matrix~\cite{STS}. Rota-Baxter operators of weight 1  are in one-to-one correspondence  with
solutions of the modified Yang-Baxter equation (see Remark
\ref{rmk:mybe}), and give rise to   factorizations of Lie algebras. Integrating such factorizations locally to the level of Lie groups provides natural solutions of a certain Hamiltonian system, which is of great interest for the people coming from integrable
systems~\cite{STS}. In a recent work \cite{GLS}, the notion of a Rota-Baxter operator on a
Lie group was introduced. In particular, one can obtain a Rota-Baxter
operator of weight 1 on a Lie algebra by taking the differentiation of a
Rota-Baxter operator on a Lie group.

Rota-Baxter operators  are
special cases of relative Rota-Baxter operators, where we distinguish
the source and the target of the operator (see Def. \ref{defi:rb-lie-algebra} and
\ref{defi:rb-lie-group}).  In this article, we further study the
differentiation and integration problem for relative Rota-Baxter
operators, and explore their cohomology theory. We find such a
separation of the source and the target makes the cohomology theory and Lie
theory actually much more clear. Differentiation on the level of
cohomology, namely a version of Van-Est theorem, is also given. We
show that differentiation and local integration work for relative
Rota-Baxter operators. Moreover, these two procedures are adjoint to
each other. Our results in particular imply that we may integrate a
Rota-Baxter operator on a Lie algebra to a local Rota-Baxter operator
on a Lie group in a functorial way. This seems to provide some
global geometrical insight to the above mentioned problem in integrable
systems. Moreover, via our integration, we are able to provide an explicit formula for the factorization problem on the level of Lie groups.

Differentiation works without much obstruction, and follows
from adapted Lie theoretical calculations. Integration, even local
integration, is however, a bit tricky: the relative Rota-Baxter
operator is not a Lie algebra or a Lie group homomorphism. To integrate such a map, we
turn to its graph, which has a relative good Lie theoretical property
but not completely compatible with exponential maps. Nevertheless,
a bit luck happens, and the non-compatibleness somehow falls in the kernel
of the operator and the local integration works.  To really achieve the global integration, we expect some topological obstruction. In fact, we have already faced this when we try to integrate an infinite dimensional Lie algebra \cite{neeb:cent-ext-gp, wz:int}, where the obstruction lies in second cohomology groups; and when we try to integrate a Lie algebroid \cite{cf} to a Lie groupoid. On the other hand, we might also use Malcev's method to extend local structures to global ones \cite{malcev}. As therein, an obstruction in algebraic terms would show up. Nevertheless, it turns out that the algebraic obstruction therein is actually  topological in nature \cite{fernandes:ass}. Thus, the local integration and some topological invariants obtained in this article should be good ingredients for the global integration in the next step.

We notice that the cohomology theory of relative Rota-Baxter
operators of weight 0 on Lie algebras, together with those on associative
algebras have been much studied recently \cite{Das,DasM,LST, TBGS}. Also there are some further progresses on Rota-Baxter operators on groups in recently works \cite{BaG,BaG2,Gon}.

A byproduct of our cohomology theory is to provide a classification of
deformation of modified $r$-matrices, which seems to be easier to
carry out from the operator point of view.

Now we describe our results in the time order: we first show that a relative Rota-Baxter operator
$B:\h\lon\g$ of weight 1 on a Lie algebra $\g$ with respect to an action
$\phi:\g\lon\Der(\h)$ gives rise to a representation of the descendent
Lie algebra $(\h,[\cdot,\cdot]_B)$ on the vector space $\g$. It turns
out that the Chevalley-Eilenberg cochain complex of the descendent Lie algebra gives rise to the
controlling algebra of the relative Rota-Baxter operator  $B$ of weight 1
\cite{T-Bai-Guo-Sheng}. In principal, the controlling algebra of an
algebraic structure together with a
differential coming from the algebraic structure itself
provides a good cohomology theory for this algebraic structure. We thus
define the
cohomology of the relative Rota-Baxter operator $B$ to be the
Chevalley-Eilenberg cohomology of the corresponding descendent Lie
algebra. As an application or verification of this cohomology theory, we apply the second
cohomology group to study infinitesimal deformations of relative
Rota-Baxter operators of weight 1 on Lie algebras and that of modified
$r$-matrices. Then parallelly,  we show that
a relative Rota-Baxter operator $\huaB:H\lon G$ on a Lie group $G$
with respect to an action $\Phi:G\lon\Aut(H)$ gives rise to a
representation of the descendent Lie group $(H,\star)$ on the vector
space $\g$. Analogously, the cohomology of a relative Rota-Baxter
operator $\huaB$ is taken to be the  cohomology of the corresponding
descendent Lie group. Moreover, we establish the Van Est map
  from the cochain complex of the relative Rota-Baxter operator
  $\huaB$ on a Lie group $G$ to the cochain complex of the relative
  Rota-Baxter operator $B$ on the corresponding Lie algebra
  $\g$. Van Est theorems hold, and this might also be viewed as a
  justification of our cohomology theory for relative Rota-Baxter
  operators on Lie groups.

In Section \ref{sec:alg}, we give the cohomology of a relative
Rota-Baxter operator of weight 1 on a Lie algebra $\g$ with respect to
an action $\phi:\g\lon\Der(\h)$ using the Chevalley-Eilenberg
cohomology of the descendent Lie algebra $(\h,[\cdot,\cdot]_B)$. As
applications, we study infinitesimal deformations of a relative
Rota-Baxter operator of weight 1 and modified $r$-matrices using the second cohomology group.

In Section \ref{sec:grp}, we give the cohomology of a relative Rota-Baxter operator on a Lie group $G$ with respect to an action $\Phi:G\lon\Aut(H)$ using the  cohomology of the descendent Lie group $(H,\star)$.

In Section \ref{sec:ve},  we first show that the Lie algebra of the
descendent Lie group $(H,\star)$ is the descendent Lie algebra
$(\h,[\cdot,\cdot]_B)$, and the differentiation of the representation
$\Theta:H\lon\gl(\g)$ of the  descendent Lie group $(H,\star)$ is the
representation $\theta:\h\lon\gl(\g)$ of the  descendent Lie algebra
$(\h,[\cdot,\cdot]_B)$. Then we achieve our differentiation operator and consequently, we  establish the Van Est map from the cochain complex of the relative Rota-Baxter operator $\huaB$ on a Lie group $G$ to the cochain complex of the relative Rota-Baxter operator $B$ of weight 1 on the corresponding Lie algebra $\g$.

In Section \ref{sec:loi}, we introduce the notion of a local relative Rota-Baxter operator on a Lie group $G$ and show that a relative Rota-Baxter operator $B:\h\lon\g$ of weight 1 on a Lie algebra can be integrated to  a local relative Rota-Baxter operator on the corresponding connected and simply connected  Lie group $G$.

In Section \ref{sec:app}, we give two applications of the integration of Rota-Baxter operators. One is to provide an explicit formula of local factorization problem and the other is to provide an integration for some special matched pairs.



\vspace{2mm}
\noindent
{\bf Acknowledgements. } We give our warmest thanks to Jiang-Hua Lu
for very helpful discussions. This research is supported by NSFC
(11922110) and DFG ZH 274/3-1.

\section{Relative Rota-Baxter operators on Lie algebras and their
  cohomology theory }\label{sec:alg}

In this section, we define the cohomology of relative Rota-Baxter operators of weight 1 on Lie algebras. As applications, we classify infinitesimal deformations of relative Rota-Baxter operators of weight 1  on Lie algebras using the second cohomology group.

\subsection{The category $\RBgh$ of relative Rota-Baxter operators on Lie algebras}

In the sequel,  $(\g, [\cdot,\cdot]_{\g})$  and $ (\h,
[\cdot,\cdot]_{\h})$ are   Lie algebras. Let $\phi: \g \to \Der(\h)$ be a Lie algebra
  homomorphism. We call $\phi$ an {\bf action} of $\g$ on $\h$.

\begin{defi}\label{defi:rb-lie-algebra}
Let  $\phi: \g\rightarrow\Der(\h)$ be an action of a Lie
algebra $(\g, [\cdot,\cdot]_{\g})$ on a Lie algebra $ (\h,
[\cdot,\cdot]_{\h})$. A linear map $B: \h\rightarrow\g$ is called a {\bf
  relative Rota-Baxter operator  of weight $1$}  on $\g$ with respect
to $(\h;\phi)$  if
\begin{equation} \label{eq:B}
[B(u), B(v)]_\g=B\Big(\phi(B(u))v-\phi(B(v))u+[u,v]_\h\Big),\quad \forall u, v\in\h.
\end{equation}
In particular, if $\g=\h$ and the action    is the adjoint representation of $\g$ on itself, then   $B$ is called a {\bf Rota-Baxter operator of weight $1$}.
\end{defi}

\begin{rmk}
In the context of Lie algebras, usually one considers relative
Rota-Baxter operators of weight $\lambda$, where $\lambda$ is a
parameter in the front of $[\cdot,\cdot]_\h$ in \eqref{eq:B}. However, in the
context of Lie groups, one can only consider relative Rota-Baxter
operators of weight $\pm1$ due to the fact that there is no linear
structure on Lie groups in general. Thus, in this paper we
consider only relative Rota-Baxter operators of weight $1$, and call
them simply relative Rota-Baxter operators.
\end{rmk}

\begin{rmk}\label{rmk:mybe}
A Rota-Baxter operator $B$ on $\g$ one-to-one corresponds to a solution of  the {\bf modified
      Yang-Baxter equation}, which in turn has important applications
    in integrable systems. We now recall some facts
    from~\cite{Iz,RS1,STS,STS2}.
 Let $R:\g\lon\g$ be a solution of the  modified Yang-Baxter equation:
\begin{equation}
	\label{eq:mybe}
	 [R(u),R(v)]_\g=R([R(u),v]_\g)+R([u,R(v)]_\g)-[u,v]_\g, \quad
         \forall u, v\in \g.
\end{equation}
Each such a solution (which we call  a {\bf modified
  $r$-matrix}) gives rise to a so-called infinitesimal
factorization of the Lie algebra $\g$. The authors integrated this infinitesimal factorization
locally into a factorization of the Lie group $G$ of $\g$. The
integrated factorization in turn
gives an elegant solution of a certain hamiltonian system for small
time $t$. The factorization itself may also be viewed as some version
of Riemann-Hilbert problem.

Under the transformation $$R=\Id+2B,$$ the operator $R$ satisfies the modified Yang-Baxter equation if and only if the operator $B$ is  a  Rota-Baxter operator of weight $1$.
\end{rmk}

\begin{ex}
Let $\g=\mathrm{up}(2;\mathbb{R})$, where $\mathrm{up}(2;\mathbb{R})$ is the set of upper triangular matrices and $\h=\mathbb{R}$. Define $B:\mathbb{R}\lon\mathrm{up}(2;\mathbb{R})$ and $\phi:\mathrm{up}(2;\mathbb{R})\lon\Der(\mathbb{R})$ by
\begin{equation*}
B(r)=\left(
       \begin{array}{cc}
         0 & r \\
         0 & 0 \\
       \end{array}
     \right),\quad \phi(\left(
                          \begin{array}{cc}
                            x & y \\
                            0 & z \\
                          \end{array}
                        \right))r=xr, \quad \forall \left(
                          \begin{array}{cc}
                            x & y \\
                            0 & z \\
                          \end{array}
                        \right) \in\mathrm{up}(2;\mathbb{R}), r\in\mathbb{R}.
\end{equation*}
Then $B$ is a relative Rota-Baxter operator on the Lie algebra $\g$ with respect to the action $(\mathbb R; \phi)$.
\end{ex}

Let  $\phi: \g\rightarrow\Der(\h)$ be an action of a Lie algebra $(\g, [\cdot,\cdot]_{\g})$ on a Lie algebra $ (\h, [\cdot,\cdot]_{\h})$. Define a skewsymmetric bracket operation $[\cdot,\cdot]_{\phi}$ on $\g\oplus \h$ by
\begin{equation}
 [x+u,y+v]_{\phi}=[x,y]_\g+\phi(x)v-\phi(y)u+[u,v]_\h,\quad \forall x,y\in\g, ~u, v\in\h.
\end{equation}
Then it is straightforward to see  that $(\g\oplus \h,[\cdot,\cdot]_{\phi})$ is a Lie algebra, and denoted by $\g\ltimes _\phi\h$.

\begin{pro}\label{graphro}
Let  $\phi: \g\rightarrow\Der(\h)$ be an action of a Lie algebra $(\g, [\cdot,\cdot]_{\g})$ on a Lie algebra $ (\h, [\cdot,\cdot]_{\h})$. Then a linear map $B: \h\rightarrow\g$ is a relative Rota-Baxter operator if and only if the graph   $Gr(B)=\{B(u)+u|u\in \h\}$ is a subalgebra of the  Lie algebra $\g\ltimes _\phi\h$.
\end{pro}
\begin{proof}
Let $B:\h\rightarrow\g$ be a linear map.
For all $u,v\in \h,$  we have
\begin{eqnarray*}
[B(u)+u,B(v)+v]_{\phi}=[B(u),B(v)]_{\g}+\phi(B(u))v-\phi(B(v))u+[u,v]_\h,
\end{eqnarray*}
which implies that the graph $Gr(B)=\{B(u)+u|u\in \h\}$ is a subalgebra of the  Lie algebra $\g\ltimes _\phi\h$ if and only if $B$ satisfies
\begin{eqnarray*}
[B(u),B(v)]_{\g}=B\Big(\phi(B(u))v-\phi(B(v))u+[u,v]_\h\Big),
\end{eqnarray*}
which means that $B$ is a relative Rota-Baxter operator.
\end{proof}

Since $\h$ and the graph
$Gr(B)$ are isomorphic as vector spaces, we have the following result.
\begin{cor}\label{newliealg}
Let $B: \h\rightarrow\g$ be a relative Rota-Baxter operator on $\g$ with respect to an action $(\h; \phi)$. Then $(\h, [\cdot,\cdot]_B)$ is a Lie algebra, called {\bf the descendent Lie algebra} of $B$, where
\begin{equation}\label{eq:desLieb}
[u,v]_B=\phi(B(u))v-\phi(B(v))u+[u,v]_{\h}, \quad \forall u, v\in\h.
\end{equation}
Moreover, $B$ is a Lie algebra homomorphism from $(\h, [\cdot,\cdot]_B)$ to $(\g, [\cdot,\cdot]_\g)$.
\end{cor}

\begin{defi}
Let $B$ and $B'$ be relative Rota-Baxter operators on a Lie algebra $\g$ with respect to the action $(\h;\phi)$. A {\bf homomorphism} from $B'$ to $B$ consists of a Lie algebra homomorphism $\psi_\g: \g\lon\g$ and a Lie algebra homomorphism $\psi_\h: \h\lon\h$ such that
\begin{eqnarray}\label{hom-rbo}
\label{hom-rbo1}\psi_\g\circ B'&=&B\circ\psi_\h,\\
\label{hom-rbo2}\psi_\h(\phi(x)v)&=&\phi (\psi_\g(x) )\psi_\h(v), \quad \forall x\in\g, v\in\h.
\end{eqnarray}
In particular, if both $\psi_\g$ and $\psi_\h$ are invertible, $(\psi_\g, \psi_\h)$ is called an isomorphism from $B'$ to $B$.
\end{defi}

In fact, \eqref{hom-rbo2} is equivalent that $(\psi_\g,\psi_\h)$ is an endomorphism of the Lie algebra $\g\ltimes _\phi\h$.

It is straightforward to see that  relative Rota-Baxter operators on a Lie algebra $\g$ with respect to an action $(\h;\phi)$ together with homomorphisms between them form a category, which we denote by $\RBgh$.

\subsection{A cohomology theory in $\RBgh$}
We first prove some statements which are important for us to build up
our  cohomology theory in a functorial way.
\begin{pro}\label{pro:rep-on-g-alg}
Let $B: \h\rightarrow\g$ be a relative Rota-Baxter operator  on $\g$ with respect to   $(\h;\phi)$.  Define a linear map $\theta: \h\rightarrow\gl(\g)$  by
\begin{equation}\label{lierepT}
\theta(u)x=B(\phi(x)u)+[B(u), x]_\g, \quad \forall x\in\g, u\in\h.
\end{equation}
Then $\theta$ is a representation of the descendent Lie algebra $(\h,[\cdot,\cdot]_B)$ on the vector space $\g$.
\end{pro}
\begin{proof}
By  the fact  that $\phi(x)\in\Der(\h)$ for all $x\in\g$ and $B: \h\rightarrow\g$ is a relative Rota-Baxter operator, for all $u, v\in\h$, we have
\begin{eqnarray*}
[\theta(u),\theta(v)]x&=&\theta(u)\theta(v)x-\theta(v)\theta(u)x\\
&=&\theta(u)(B(\phi(x)v))+\theta(u)[B(v),x]_\g-\theta(v)(B(\phi(x)u))-\theta(v)[B(u),x]_\g\\
&=&B(\phi(B(\phi(x)v))u)+[B(u), B(\phi(x)v)]_\g+B(\phi([B(v),x]_\g)u)+[B(u), [B(v), x]_\g]_\g\\
&&-B(\phi(B(\phi(x)u))v)-[B(v), B(\phi(x)u)]_\g-B(\phi([B(u),x]_\g)v)-[B(v), [B(u), x]_\g]_\g\\
&=&B\Big(\phi(B(u))\phi(x)v-\phi(B(\phi(x)v))u+[u,\phi(x)v]_\h\Big)\\
&&+B(\phi(B(\phi(x)v))u)+B(\phi([B(v),x]_\g)u)+[B(u), [B(v), x]_\g]_\g\\
&&-B\Big(\phi(B(v))\phi(x)u-\phi(B(\phi(x)u))v+[v,\phi(x)u]_\h\Big)\\
&&-B(\phi(B(\phi(x)u))v)-B(\phi([B(u),x]_\g)v)-[B(v), [B(u), x]_\g]_\g\\
&=&B(\phi(B(u))\phi(x)v)+B(\phi([B(v),x]_\g)u)-B(\phi(B(v))\phi(x)u)-B(\phi([B(u),x]_\g)v)\\
&&+B([u,\phi(x)v]_\h)-B([v,\phi(x)u]_\h)+[B(u), [B(v), x]_\g]_\g-[B(v), [B(u), x]_\g]_\g\\
&=&B\Big(\phi(x)\phi(B(u))v-\phi(x)\phi(B(v))u+\phi(x)[u,v]_\h\Big)+[[B(u),B(v)]_\g, x]_\g\\
&=&B(\phi(x)[u,v]_B)+[B([u,v]_B), x]_\g\\
&=&\theta([u,v]_B)x.
\end{eqnarray*}
Thus $\theta$ is a representation of the descendent Lie algebra $(\h,[\cdot,\cdot]_B)$ on the vector space $\g$.
\end{proof}

\begin{pro}
Let $B: \h\rightarrow\g$ be a relative Rota-Baxter operator  on $\g$ with respect to $(\h;\phi)$.  Then
\begin{equation*}
\phi(x)[u,v]_{B}=[\phi(x)u,v]_{B}+[u,\phi(x)v]_{B}+\phi(\theta(v)x)u-\phi(\theta(u)x)v,\quad \forall u,v\in \h, x\in \g.
\end{equation*}
\end{pro}
\begin{proof}
For any $u,v\in\h, x\in\g$, by Proposition \ref{pro:rep-on-g-alg}, we obtain
\begin{eqnarray*}
&&\phi(x)[u,v]_{B}-([\phi(x)u,v]_{B}+[u,\phi(x)]_{B})\\
&=&\phi(x)\phi(B(u))v-\phi(x)\phi(B(v))u+\phi(x)[u,v]_{\h}\\
&&-(\phi(B(\phi(x)u))v-\phi(B(v))\phi(x)u+[\phi(x)u,v]_{\h})\\
&&-(\phi(B(u))\phi(x)v-\phi(B(\phi(x)v))u+[u,\phi(x)v]_{\h})\\
&=&[\phi(B(v)),\phi(x)]u+\phi(B(\phi(x)v))u-[\phi(B(u)),\phi(x)]v-\phi(B(\phi(x)u))v\\
&=&\phi([B(v),x]_{\g}+B(\phi(x)v))u-\phi([B(u),x]_{\g}+B(\phi(x)u))v\\
&=&\phi(\theta(v)x)u-\phi(\theta(u)x)v,
\end{eqnarray*}
which finishes the proof.
\end{proof}

Let $\dM_{\CE}^B: \Hom(\wedge^{k}\h, \g)\rightarrow\Hom(\wedge^{k+1}\h, \g)$ be the corresponding Chevalley-Eilenberg coboundary operator of the descendent Lie algebra $(\h,[\cdot,\cdot]_B)$ with coefficients in $(\g;\theta)$. More precisely, for all $f\in\Hom(\wedge^{k}\h, \g)$ and $u_1,\cdots, u_{k+1}\in\h$, we have
\begin{eqnarray}
\label{eq:dce}&&\dM_{\CE}^Bf(u_{1},\cdots, u_{k+1})\\
\nonumber&=&\sum_{i=1}^{k+1}(-1)^{i+1}\theta(u_i)f(u_1, \cdots, \hat{u_i}, \cdots, u_{k+1})+\sum_{i<j}(-1)^{i+j}f([u_i, u_j]_B, u_1, \cdots, \hat{u_i}, \cdots, \hat{u_j}, \cdots, u_{k+1})\\
\nonumber&=&\sum_{i=1}^{k+1}(-1)^{i+1}B(\phi(f(u_1, \cdots, \hat{u_i}, \cdots, u_{k+1}))u_i)+\sum_{i=1}^{k+1}(-1)^{i+1}[B(u_i), f(u_1, \cdots, \hat{u_i}, \cdots, u_{k+1})]_\g\\
\nonumber&&+\sum_{i<j}(-1)^{i+j}f(\phi(B(u_i))u_j-\phi(B(u_j))u_i+[u_i,u_j]_{\h}, u_1, \cdots, \hat{u_i}, \cdots, \hat{u_j}, \cdots, u_{k+1}).
\end{eqnarray}

\begin{defi}\label{deficorbal}
Let $B:\h\lon\g$ be a relative Rota-Baxter operator on the Lie algebra $\g$ with respect to an action $(\h; \phi)$. Define the space of $k$-cochains $C^k(B)$ by $C^k(B)=\Hom(\wedge^{k-1}\h,\g)$. Denote by $C^*(B)=\oplus_{k=1}^{+\infty}C^k(B)$. The cohomology of the   cochain complex $(C^*(B), \dM_{\CE}^B)$  is defined to be the {\bf   cohomology  for the relative Rota-Baxter operator $B$}.
\end{defi}

Denote by $H^k(B)$ the $k$-th cohomology group.

It implies in particular that $x\in\g$ is a 1-cocycle  if and only if
$
B\circ\phi(x)=\ad_x\circ B,
$
and $f\in \Hom(\h, \g)$ is a 2-cocycle if and only if
\begin{eqnarray*}
&&B(\phi(f(u_2))u_1)+[B(u_1), f(u_2)]_\g-B(\phi(f(u_1))u_2)-[B(u_2), f(u_1)]_\g\\
&=&f(\phi(B(u_1))u_2-\phi(B(u_2))u_1+[u_1,u_2]_\h).
\end{eqnarray*}

\begin{rmk}
The controlling algebra of relative
Rota-Baxter operators of weight $1$ on Lie algebras was given in \cite{T-Bai-Guo-Sheng}.  The cohomology theory established here coincides with the one given by the controlling algebra. As inspired by the slogan promoted by Deligne:
{\text
deformations of a given algebraic or geometric structure are governed
by a differential graded Lie algebra (DGLA), or more generally, by an
$L_\infty$-algebra}. This DGLA or
$L_\infty$-algebra, is called the {\bf controlling algebra} of the given
algebraic or geometric structure. Moreover, given the close relation
between the cohomology theory and the deformation theory, we may use the
controlling algebra to define a cohomology theory for the given algebraic or
geometric structure.
\end{rmk}

We need the following statement to prove the functoriality of our
cohomology theory.
\begin{pro}\label{prohomrepalg}
Let $B$ and $B'$ be relative Rota-Baxter operators on a Lie algebra $\g$ with respect to an action $(\h;\phi)$ and $(\psi_\g, \psi_\h)$ be a homomorphism from $B'$ to $B$.

\begin{itemize}
  \item[{\rm(i)}]  $\psi_\h$ is also a Lie algebra homomorphism from
the descendent Lie algebra $(\h, [\cdot,\cdot]_{B'})$ of $B'$ to the descendent Lie algebra $(\h, [\cdot,\cdot]_{B})$ of $B$;

   \item[{\rm(ii)}] The  induced representation $(\g;\theta)$ of the Lie algebra $(\h,[\cdot,\cdot]_B)$ and the induced representation $(\g;\theta')$ of the Lie algebra $(\h,[\cdot,\cdot]_{B'})$ satisfy the following relation:
$$
\psi_\g\circ \theta'(u)=\theta(\psi_\h(u)) \circ \psi_\g,\quad \forall u\in \h.
$$
That is, for all $u\in \h$,   the following diagram commutes:
$$
\xymatrix{
  \g \ar[d]_{\theta'(u)} \ar[r]^{\psi_\g}
                & \g \ar[d]^{\theta(\psi_\h(u))}  \\
    \g \ar[r]_{\psi_\g}
                & \g            .}
$$
\end{itemize}
\end{pro}
\begin{proof}
By   \eqref{hom-rbo1}, \eqref{hom-rbo2} and the fact that $\psi_\h$ is a Lie algebra homomorphism, we have
\begin{eqnarray*}
\psi_\h([u,v]_{B'})&=&\psi_\h\Big(\phi(B'(u))v-\phi(B'(v))u+[u,v]_\h\Big)\\
&=&\phi\Big(B(\psi_\h(u))\Big)\psi_\h(v)-\phi\Big(B(\psi_\h(v))\Big)\psi_\h(u)+[\psi_\h(u), \psi_\h(v)]_\h\\
&=&[\psi_\h(u),\psi_\h(v)]_{B}, \quad \forall u, v\in\h,
\end{eqnarray*}
which implies that $\psi_\h $ is a homomorphism between the descendent Lie algebras.

By \eqref{hom-rbo1}, \eqref{hom-rbo2} and \eqref{lierepT}, for all $x\in\g, u\in\h$, we have
\begin{eqnarray*}
\psi_\g(\theta'(u)x)&=&\psi_\g(B'(\phi(x)u))+\psi_\g([B'(u),x]_\g)=B(\psi_\h(\phi(x)u))+[\psi_\g(B'(u)),\psi_\g(x)]_\g\\
&=&B(\phi(\psi_g(x))\psi_\h(u))+[B(\psi_\h(u)),\psi_\g(x)]_\g=\theta(\psi_\h(u))\psi_\g(x).
\end{eqnarray*}
We finish the proof.
\end{proof}

Let $B$ and $B'$ be relative Rota-Baxter operators on $\g$ with respect to an action $(\h;\phi)$, and $(\psi_\g, \psi_\h)$ a homomorphism from $B'$ to $B$ in which $\psi_\h$ is invertible.  Define a map $p: C^{k}(B')\lon C^{k}(B)$ by
\begin{equation*}
p(\omega)(u_1,\cdots,u_{k-1})=\psi_{\g}(\omega(\psi_\h^{-1}(u_1), \cdots,\psi_\h^{-1}(u_{k-1}))), \quad \forall u_i\in \h.
\end{equation*}

\begin{thm}
With  above notations, $p$ is a cochain map from the cochain complex $(C^*(B'), \dM_{\CE}^{B'})$ to the cochain complex $(C^*(B), \dM_{\CE}^B)$. Consequently, it induces a homomorphism $p_*$ from   the cohomology group $H^{k}(B')$  to  $H^{k}(B)$ for all $k\geq 1$.
\end{thm}
\begin{proof}
  For all $\omega\in C^{k}(B')$, by Proposition \ref{prohomrepalg}, we have
\begin{eqnarray*}
&&\dM_{\CE}^{B}(p(\omega))(u_1,u_2,\cdots,u_{k})\\
&=&\sum_{i=1}^{k}(-1)^{i+1}\theta(u_i)p(\omega)(u_1, \cdots, \hat{u_i}, \cdots, u_{k})\\
&&+\sum_{i<j}(-1)^{i+j}p(\omega)([u_i, u_j]_B, u_1, \cdots, \hat{u_i}, \cdots, \hat{u_j}, \cdots, u_{k})\\
&=&\sum_{i=1}^{k}(-1)^{i+1}\theta(u_i)\psi_{\g}\big(\omega(\psi_{\h}^{-1}(u_1), \cdots, \hat{\psi_{\h}^{-1}(u_i)}, \cdots, \psi_{\h}^{-1}(u_{k}))\big)\\
&&+\sum_{i<j}(-1)^{i+j}\psi_\g\Big(\omega(\psi_\h^{-1}([u_i, u_j]_B),\psi_\h^{-1}(u_1), \cdots, \hat{\psi_\h^{-1}(u_i)}, \cdots, \hat{\psi_\h^{-1}(u_j)}, \cdots, \psi_\h^{-1}(u_{k}))\Big)\\
&=&\sum_{i=1}^{k}(-1)^{i+1}\psi_{\g}\big(\theta'(\psi_\h^{-1}(u_i))\omega(\psi_{\h}^{-1}(u_1), \cdots, \hat{\psi_{\h}^{-1}(u_i)}, \cdots, \psi_{\h}^{-1}(u_{k}))\big)\\
&&+\sum_{i<j}(-1)^{i+j}\psi_\g\Big(\omega([\psi_\h^{-1}(u_i), \psi_\h^{-1}(u_j)]_{B^{'}},\psi_\h^{-1}(u_1), \cdots, \hat{\psi_\h^{-1}(u_i)}, \cdots, \hat{\psi_\h^{-1}(u_j)}, \cdots, \psi_\h^{-1}(u_{k}))\Big)\\
&=&\psi_{\g}\big(\dM_{\CE}^{B'}\omega(\psi_\h^{-1}(u_1),\cdots,\psi_\h^{-1}(u_{k}))\big)\\
&=&p(\dM_{\CE}^{B'}\omega)(u_1,\cdots,u_{k}),\quad \forall u_i\in \h.
\end{eqnarray*}
Thus $p$ is a cochain map. Consequently it  induces a homomorphism $p_*$ from   the cohomology group $H^{k}(B')$  to  $H^{k}(B)$ for all $k\geq 1$.
\end{proof}

\subsection{Cohomology theory and deformations in $\RBgh$ and modified
  $r$-matrices}
In this subsection, we show that the second cohomology group
classifies the infinitesimal deformations of  relative Rota-Baxter
operators. This may be viewed as a justification of the cohomology
theory that we established in the previous subsection. In the end, we use our cohomology theory to study the
deformations of modified $r$-matrices as an application.

\begin{defi}
Let $B:\h\lon\g$ be a relative Rota-Baxter operator on $\g$ with respect to an action $(\h; \phi)$ and $\hat{B}: \h\rightarrow\g$ be a linear map. If $B_{t}=B+t\hat{B}$ is still a relative Rota-Baxter operator on the Lie algebra $\g$ with respect to the action $(\h; \phi)$ for all $t$, we say that $\hat{B}$ generates a {\bf one-parameter infinitesimal deformation} of the relative Rota-Baxter operator $B$.
\end{defi}

 Since $B_{t}=B+t\hat{B}$ is a relative Rota-Baxter operator on the Lie algebra $\g$ with respect to the representation $(\h; \phi)$ for all $t$, we have
\begin{equation}\label{rbd-rb}
[\hat{B}(u),\hat{B}(v)]_\g=\hat{B}(\phi(\hat{B}(u))v)-\hat{B}(\phi(\hat{B}(v))u),
\end{equation}
and
\begin{eqnarray}\label{rbd-1c}
&&[\hat{B}(u),B(v)]_\g+[B(u),\hat{B}(v)]_\g\\
\nonumber&=&B(\phi(\hat{B}(u))v)+\hat{B}(\phi(B(u))v)-B(\phi(\hat{B}(v))u)-\hat{B}(\phi(B(v))u)+\hat{B}([u,v]_\h).
\end{eqnarray}
Note that equation \eqref{rbd-rb} states that $\hat{B}$ is a relative Rota-Baxter operator of wight $0$ on the Lie algebra $\g$ with respect to the representation $(\h; \phi)$. The equation \eqref{rbd-1c} states that $\hat{B}$ is a $2$-cocycle of the relative Rota-Baxter operator $B$.

\begin{defi}
Let $B: \h\rightarrow\g$ be a relative Rota-Baxter operator  on $\g$ with respect to   $(\h;\phi)$. Two one-parameter infinitesimal deformations $B_t^1=B+t\hat{B}_1$ and $B_t^2=B+t\hat{B}_2$ are said to be {\bf equivalent} if there exists an $x\in\g$ such that $(\Id_\g+t\ad_x, \Id_\h+t\phi(x))$ is a homomorphism from $B_t^1$ to $B_t^2$. In particular, a one-parameter infinitesimal deformation $B_t$ of a relative Rota-Baxter operator $B$ is said to be {\bf trivial} if there exists an $x\in\g$ such that $(\Id_\g+t\ad_x, \Id_\h+t\phi(x))$ is a homomorphism from $B_t$ to $B$.
\end{defi}

Let $(\Id_\g+t\ad_x, \Id_\h+t\phi(x))$ be a  homomorphism from $B_t^1$ to $B_t^2$. Then we have
\begin{equation*}
(\Id_\g+t\ad_x)(B+t\hat{B}_1)u=(B+t\hat{B}_2)(\Id_\h+t\phi(x))u, \quad \forall u\in\h,
\end{equation*}
which implies
\begin{eqnarray}\label{sameclass}
\nonumber[x, \hat{B}_1(u)]_\g&=&\hat{B}_2(\phi(x)u),\\
\hat{B}_1(u)-\hat{B}_2(u)&=&B(\phi(x)u)+[B(u), x]_\g \quad \forall x\in\g, u\in\h.
\end{eqnarray}

By \eqref{sameclass}, we have

\begin{equation*}
\hat{B}_1-\hat{B}_2=\dM_{\CE}^{B}x,
\end{equation*}
where $\dM_{\CE}^{B}$ is given by \eqref{eq:dce}. Thus we have

\begin{thm}\label{thm:rb-deform}
Let $B: \h\rightarrow\g$ be a relative Rota-Baxter operator  on $\g$ with respect to   $(\h;\phi)$. If two one-parameter infinitesimal deformations $B_t^1=B+t\hat{B}_1$ and $B_t^2=B+t\hat{B}_2$ are equivalent, then $\hat{B}_1$ and $\hat{B}_2$ are in the same cohomology class of $H^{2}(B)$ defined in Definition \ref{deficorbal}.
\end{thm}

Since modified $r$-matrices and Rota-Baxter operators of weight $1$
are related via the formula $R=\Id+2B$ (see Remark \ref{rmk:mybe}), it
follows that our cohomology theory can also control the infinitesimal
deformations of modified $r$-matrices. More details will be provided
in a future work.
\begin{defi}
Given a Lie algebra $\g$, let $R$ be a modified $r$-matrix satisfying
Eq. \eqref{eq:mybe}. If $R_t=R+t\hat{R}$ is still a modified $r$-matrix for
all $t$, we say that $\hat{R}$ generates a {\bf one-parameter
  infinitesimal deformation} of $R$.
\end{defi}

\begin{defi}
Let   $R_t^1=R+t\hat{R}_1$ and $R_t^2=R+t\hat{R}_2$ be two infinitesimal deformations of a modified $r$-matrix $R$. They are said to be {\bf equivalent} if there exists an $x\in\g$ such that $\Id_\g+t\ad_x$ is a homomorphism from $R_t^1$ to $R_t^2$, i.e. $\Id_\g+t\ad_x$ is a Lie algebra endomorphism and $$(\Id_\g+t\ad_x)\circ R_t^1=R_t^2\circ(\Id_\g+t\ad_x).$$
\end{defi}

\begin{thm}
 Let   $R_t^1=R+t\hat{R}_1$ and $R_t^2=R+t\hat{R}_2$ be two equivalent infinitesimal deformations of a modified $r$-matrix $R$. Then $\hat{R}_1$ and $\hat{R}_2$ are in the same cohomology class of $H^{2}(B)$ defined in Definition \ref{deficorbal} for $B=\frac{1}{2}(R-\Id)$.
\end{thm}

\begin{proof}
  It is obvious that  an infinitesimal deformation of a Rota-Baxter
  operator $B_t=B+t\hat{B}$ gives rise to an infinitesimal deformation
  of a modified $r$-matrix $R_t=R+t(2\hat{B})$, with
  $R=\Id+2B$. Conversely, an infinitesimal deformation of a modified
  $r$-matrix $R_t=R+t\hat{R}$ gives rise to an infinitesimal
  deformation of a Rota-Baxter operator $B_t=B+t(\frac{\hat{R}}{2})$,
  with $B=\frac{1}{2}(R-\Id)$. Therefore $\hat{R}$ is a
  2-cocycle in $(C^*(B), \dM^B_{\CE})$ by Theorem \ref{thm:rb-deform}. Similarly,
  $R_t^1=R+t\hat{R}_1$ and $R_t^2=R+t\hat{R}_2$ are equivalent
  infinitesimal deformations if and only if their corresponding
  Rota-Baxter operators are equivalent. Therefore by Theorem \ref{thm:rb-deform}, both $\hat{R}_1$ and
  $\hat{R}_2$ are cocycles in  $(C^*(B), \dM^B_{\CE})$ and they are differed by
  a coboundary therein.
\end{proof}

\emptycomment{
1. Calculation: $\hat{R}$ must be a usual $r$-matrix, and it corresponds
to a weight 0 RB operator $\hat{B}$ which is a 2-cocycle in $H^2(B)$,
where $R=\Id + 2B$.  \\
2. the second cohomology group given in Definition \ref{deficorbal} for a Rota-Baxter operator $B$ of weight $1$ can also control infinitesimal deformations of the  modified $r$-matrix  $R=\Id+2B$.}

\section{Relative Rota-Baxter operators on Lie groups and their
  cohomology theory}\label{sec:grp}

In this section, we introduce the cohomology theory of relative Rota-Baxter operators on Lie groups. In the sequel, $(G, e_G, \cdot_G)$ and $(H, e_H, \cdot_H)$ are always Lie groups.

\subsection{The category $\RBGH$ of relative Rota-Baxter operators on
  Lie groups}
\begin{defi} \label{defi:rb-lie-group}
Let  $\Phi: G\rightarrow \Aut(H)$ be an action of $G$ on $H$. A smooth map $\huaB: H\rightarrow G$ is called a {\bf relative Rota-Baxter operator} if
\begin{equation}\label{rRBg}
\huaB(h_1)\cdot_{G}\huaB(h_2)=\huaB(h_1\cdot_{H}\Phi(\huaB(h_1))h_2), \quad \forall h_1, h_2\in H.
\end{equation}
\end{defi}

In particular if the action $\Phi$ is the adjoint action $\Ad$ of the Lie group $G$ on itself,   it reduces to a {\bf   Rota-Baxter operator} introduced in \cite{GLS}.


\begin{ex}\label{lem:RBPG}{\rm (\cite{GLS})}
Let $G$ be a Lie group and $G_{+},G_{-}$ be two subgroups such that $G=G_{+} G_{-}$ and $G_+\cap G_-=\{e\}$. Define $\huaB:G\to G$ by
$$\huaB(g)=g_{-}^{-1},\quad \forall~ g=g_{+}g_{-},\quad \mbox{where } g_{+}\in G_{+},g_-\in G_-.$$ Then $\huaB$ is a Rota-Baxter operator.
\end{ex}


Let $G$ be a connected compact real Lie group and $(\mathbb{R}, \Phi)$ be a one-dimensional representation. Thus for any $g\in G$, it follows that $|\Phi(g)|=1$. Moreover, $G$ is a connected Lie group, thus $\Phi(g)=1$ for any $g\in G$. Thus a smooth map $\huaB: \mathbb{R}\lon G$ is a relative Rota-Baxter operator with respect to the action $\Phi$ if and only if $\huaB$ is a group homomorphism from $\mathbb{R}$ to $G$.
This sort of phenomenon also extends to the non-compact case in the following example:
\begin{ex}
Let $(G, e_G, \cdot_G)=(\mathrm{UP}(2;\mathbb{R}), I, \cdot)$, where $\mathrm{UP}(2;\mathbb{R})$ is the set of invertible upper triangular matrices and
$H=\mathbb R$.
Define $\huaB: \mathbb R\lon \mathrm{UP}(2;\mathbb{R})$ and $\Phi: \mathrm{UP}(2;\mathbb{R})\lon \Aut(\mathbb R)$ by
\begin{equation*}
\huaB(r)
=\left(
    \begin{array}{cc}
      1 & r \\
      0 & 1 \\
    \end{array}
  \right), \quad \Phi(\left(
    \begin{array}{cc}
      a & b \\
      0 & c \\
    \end{array}
  \right))r=ar, \quad \forall \left(
    \begin{array}{cc}
      a & b \\
      0 & c \\
    \end{array}
  \right)\in\mathrm{UP}(2;\mathbb{R}), r\in \mathbb R.
\end{equation*}
Then $\huaB$ is a relative Rota-Baxter operator on the Lie group $G$ with respect to the action $(\mathbb R; \Phi)$.
\end{ex}

Given an action $\Phi:G\to \Aut(H)$, we have the semi-direct product Lie group $G\ltimes_\Phi H$, with
multiplication $\cdot_\Phi$ given by
\begin{equation*}
(g_1,h_1)\cdot_{\Phi}(g_2,h_2)=(g_1\cdot_Gg_2, h_1\cdot_H\Phi(g_1)h_2), \quad \forall g_i\in G, h_i\in H, i=1,2.
\end{equation*}

\begin{pro}\label{graphgroup}
Let  $\Phi: G\to\Aut(H)$ be an action of a Lie group $G$ on a Lie group $ H$. Then a smooth map $\huaB: H\rightarrow G$ is a relative Rota-Baxter operator if and only if the graph   $Gr(\huaB)=\{(\huaB(h),h)|h\in H\}$ is a Lie subgroup of the  Lie group $G\ltimes_\Phi H$.
\end{pro}
\begin{proof}
Let $\huaB: H\rightarrow G$ be a smooth map.
For all $h_1,h_2\in H,$  we have
\begin{eqnarray*}
(\huaB(h_1),h_1)\cdot_{\Phi}(\huaB(h_2),h_2)=(\huaB(h_1)\cdot_G\huaB(h_2), h_1\cdot_H\Phi(\huaB(h_1))h_2),
\end{eqnarray*}
which implies that the graph $Gr(\huaB)$ is a Lie subgroup of the  Lie group $G\ltimes_\Phi H$ if and only if $\huaB$ satisfies
\begin{eqnarray*}
\huaB(h_1)\cdot_G\huaB(h_2)=\huaB( h_1\cdot_H\Phi(\huaB(h_1))h_2),
\end{eqnarray*}
which means that $\huaB$ is a relative Rota-Baxter operator.
\end{proof}

For any $g\in G$, we have $\Phi(g)e_{H}=\Phi(g)(e_{H}\cdot_{H}e_{H})=\Phi(g)e_{H}\cdot_{H}\Phi(g)e_{H}$, thus $\Phi(g)e_{H}=e_{H}$.

\begin{pro}\label{pro:dec-gp}
Let $\huaB: H\rightarrow G$ be  a relative Rota-Baxter operator on $G$ with respect to an action $(H;\Phi)$.  Then $(H, e_H, \star)$ is a Lie group, called {\bf the descendent Lie group of $\huaB$} and denoted by $H_\huaB$, where
\begin{equation}\label{group}
h_1\star h_2=h_1\cdot_H\Phi(\huaB(h_1))h_2, \quad \forall h_1, h_2\in H.
\end{equation}
For $h\in H$, its inverse $h^\dag$ with respect to the multiplication $\star$ is given by
\begin{equation} \label{eq:dag}
h^\dag=\Phi(\huaB(h)^{-1})h^{-1},
\end{equation}
where $h^{-1}$ is the inverse element of $h$ in $(H, e_H, \cdot_H)$.

Moreover, $\huaB: (H, e_H, \star)\rightarrow (G, e_G, \cdot_{G})$ is a Lie group homomorphism.
\end{pro}
\begin{proof}
It is obvious that $e_H\star h=\Phi(e_H)h=h$ and $h\star e_H=h\cdot_H\Phi(\huaB(h))e_H=h$. Thus $e_{H}$ is the unit   in $(H, e_{H}, \star)$.

For any $h\in H$, by \eqref{rRBg}, we have $\huaB(h)^{-1}=\huaB(\Phi(\huaB(h)^{-1})h^{-1})$. By \eqref{group}, we have
\begin{eqnarray}
\label{eq:proof-inv}\Phi(\huaB(h)^{-1})h^{-1}\star h&=&\Phi(\huaB(h)^{-1})h^{-1}\cdot_H\Phi(\huaB(\Phi(\huaB(h)^{-1})h^{-1}))h\\
\nonumber&=&\Phi(\huaB(\Phi(\huaB(h)^{-1})h^{-1}))h^{-1}\cdot_H\Phi(\huaB(\Phi(\huaB(h)^{-1})h^{-1}))h\\
\nonumber&=&\Phi(\huaB(\Phi(\huaB(h)^{-1})h^{-1}))(h^{-1}\cdot_{H}h)\\
\nonumber&=&e_{H}.
\end{eqnarray}
In a similar way, we have $h\star\Phi(\huaB(h)^{-1})h^{-1}=e_H.$ Thus $\Phi(\huaB(h)^{-1})h^{-1}$ is the inverse element of $h$ in $(H, e_H, \star)$.

Finally, we prove that $\star$ is associative. For all $h_1, h_2$ and $h_3\in H$, by \eqref{rRBg} and \eqref{group}, we have
\begin{eqnarray}
\label{eq:proof-ass}(h_{1}\star h_{2})\star h_{3}&=&(h_{1}\cdot_H\Phi(\huaB(h_{1}))h_{2})\star h_3\\
\nonumber&=&h_{1}\cdot_H\Phi(\huaB(h_{1}))h_{2}\cdot_H\Phi(\huaB(h_{1}\cdot_H\Phi(\huaB(h_{1}))h_{2}))h_3\\
\nonumber&=&h_{1}\cdot_H\Phi(\huaB(h_{1}))h_{2}\cdot_H\Phi(\huaB(h_1)\cdot_G\huaB(h_2))h_3\\
\nonumber&=&h_{1}\cdot_H\Phi(\huaB(h_1))(h_{2}\cdot_H\Phi(\huaB(h_2))h_3)\\
\nonumber&=&h_1\star(h_2\star h_3).
\end{eqnarray}
Thus $(H, e_H, \star)$ is a Lie group. By \eqref{rRBg} and $\huaB(e_H)=e_G$, $\huaB$ is a Lie group homomorphism from $(H, e_H, \star)$ to $(G, e_G, \cdot_G)$.
\end{proof}

\begin{defi}
Let $\huaB$ and $\huaB'$ be relative Rota-Baxter operators on a Lie group $G$ with respect to an action $(H;\Phi)$. A {\bf homomorphism} from $\huaB'$ to $\huaB$ consists of a Lie group homomorphism $\Psi_G: G\lon G$ and a Lie group homomorphism $\Psi_H: H\lon H$ such that
\begin{eqnarray}
\label{hom-rbo-gp1}\huaB\circ\Psi_H&=&\Psi_G\circ\huaB',\\
\label{hom-rbo-gp2}\Psi_H(\Phi(g)h)&=&\Phi(\Psi_G(g))\Psi_H(h), \quad \forall g\in G, h\in H.
\end{eqnarray}
In particular, if both $\Psi_G$ and $\Psi_H$ are invertible, $(\Psi_G, \Psi_H)$ is called an isomorphism from $\huaB'$ to $\huaB$.
\end{defi}
In fact, \eqref{hom-rbo-gp2} is equivalent to the fact that $(\Psi_G,\Psi_H)$ is an endomorphism of the Lie group $G\ltimes_\Phi H$.

It is clear that relative Rota-Baxter operators on a Lie group $G$
with respect to an action $(H;\Phi)$ together with
homomorphisms between them form a category, which we denote by
$\RBGH$.

\subsection{A cohomology theory in $\RBGH$}

To establish a parallel cohomology theory in $\RBGH$, we need to  first
find an action of the descendent Lie group on the vector space $\g$.

\begin{lem}\label{actionlem}
Let  $\Phi: G\rightarrow \Aut(H)$ be an action of $G$ on $H$. Then for all $h_1, h_2\in H$, we have
\begin{equation*}
\Phi(g)(h_1\star h_2)=\Phi(g)h_1\cdot_{H}\Phi(g\cdot_G\huaB(h_1))h_2.
\end{equation*}
\end{lem}
\begin{proof}
It is straightforward.
\end{proof}

\begin{thm}\label{groupact}
Let  $\huaB$ be a relative Rota-Baxter operator on $G$ with respect to an action $(H;\Phi)$. Define $\Theta: H\rightarrow \text{Diff}(G)$ by
\begin{equation}\label{eqactionH}
\Theta(h)g=(\huaB(\Phi(g)h^{\dag}))^{-1}\cdot_{G}g\cdot_{G}\huaB(h^{\dag}), \quad \forall g\in G, h\in H.
\end{equation}
Then $\Theta$ is an action of the descendent Lie group $(H, e_{H}, \star)$ on the manifold $G$.
\end{thm}
\begin{proof}
For all $h_1, h_2\in H$ and $g\in G$, by Lemma \ref{actionlem}, we have
\begin{eqnarray*}
&&\Theta(h_1)\Theta(h_2)g\\
&=&\Theta(h_{1})\Big(\huaB(\Phi(g)h_2^{\dag})^{-1}\cdot_Gg\cdot_G\huaB(h_2^{\dag})\Big)\\
&=&\huaB(\Phi\Big(\huaB(\Phi(g)h_2^{\dag})^{-1}\cdot_Gg\cdot_G\huaB(h_2^{\dag})\Big)h_1^{\dag})^{-1}\cdot_G\huaB(\Phi(g)h_2^{\dag})^{-1}\cdot_Gg\cdot_G\huaB(h_2^{\dag})\cdot_G\huaB(h_1^{\dag})\\
&=&\Big(\huaB(\Phi(g)h_2^{\dag})\cdot_G\huaB(\Phi\Big(\huaB(\Phi(g)h_2^{\dag})^{-1}\cdot_Gg\cdot_G\huaB(h_2^{\dag})\Big)h_1^{\dag})\Big)^{-1}\cdot_Gg\cdot_G\huaB(h_2^{\dag}\star h_1^{\dag})\\
&=&\huaB\Big(\Phi(g)h_2^{\dag}\cdot_{H}\Phi(\huaB(\Phi(g)h_2^{\dag}))\Phi\Big(\huaB(\Phi(g)h_2^{\dag})^{-1}\cdot_Gg\cdot_G\huaB(h_2^{\dag})\Big)h_{1}^{\dag}\Big)^{-1}\cdot_Gg\cdot_G\huaB(h_2^{\dag}\star h_1^{\dag})\\
&=&\huaB\Big(\Phi(g)h_2^{\dag}\cdot_{H}\Phi\Big(g\cdot_G\huaB(h_2^{\dag})\Big)h_{1}^{\dag}\Big)^{-1}\cdot_Gg\cdot_G\huaB(h_2^{\dag}\star h_1^{\dag})\\
&=&\Big(\huaB(\Phi(g)(h_2^{\dag}\star h_{1}^{\dag}))\Big)^{-1}\cdot_Gg\cdot_G\huaB(h_2^{\dag}\star h_1^{\dag})\\
&=&\Theta(h_{1}\star h_{2})g.
\end{eqnarray*}
Thus, $\Theta$ is an action of the descendent Lie group $(H, e_{H}, \star)$ on the manifold $G$.
\end{proof}

  Since $\Theta(h)\in\text{Diff}(G)$ for all $h\in H$ and $\Theta(h)e_G=e_G$, then $\Theta(h)_{*e_G}: \g\rightarrow\g$ is an isomorphism of vector spaces. By $\Theta(h_1\star h_2)=\Theta(h_1)\Theta(h_2)$, we have $\Theta(h_1\star h_2)_{*e_G}=\Theta(h_1)_{*e_G}\Theta(h_2)_{*e_G}$. Thus we obtain a Lie group homomorphism from the descendent Lie group $(H, e_H, \star)$ to $GL(\g)$, which is also denoted by $\Theta:H\lon GL(\g)$. Consequently, we have the following result which plays important roles in our definition of the cohomology of relative Rota-Baxter operators on Lie groups.

  \begin{lem}\label{lem:rep-on-g}
With the above notations,  $\Theta:H\lon GL(\g)$ is an action of the descendent Lie group $(H, e_H, \star)$ on the vector space $\g$.
 \end{lem}

  Now we are ready to define a cohomology theory for relative
 Rota-Baxter operators on Lie groups. First we  recall a standard
 version of the smooth  cohomology
 of a Lie group $G$ with coefficients in a $G$-module $A$ with a $G$-action
 $\Pi$ (see e.g. \cite{Wei}). 
An $n$-cochain is a smooth map $$\alpha_n:
\underbrace{G\times\cdots\times G}_n\lon A.$$  The set of $n$-cochains
forms an abelian group, which will be denoted by $C^{n}(G, A)$. The differential $\dM: C^{n}(G, A)\rightarrow C^{n+1}(G, A)$ is defined by
\begin{eqnarray*}
\dM(\alpha_n)(g_1,\cdots,g_n,g_{n+1})&=&\Pi(g_1)\alpha_n(g_2, \cdots, g_n, g_{n+1})+(-1)^{n+1}\alpha_{n}(g_1,\cdots,g_n)\\
&&+\sum_{i=1}^{n}(-1)^{i}\alpha_n(g_1,\cdots,g_{i-1}, g_i\cdot_{G}g_{i+1}, g_{i+2}, \cdots g_{n+1}).
\end{eqnarray*}

We consider the Lie group  $(H, e_H, \star)$ and its module $\g$  via the
action $\Theta:H\lon GL(\g)$ given in Lemma \ref{lem:rep-on-g}.
Denote by $C^k(\huaB)=C^{k-1}(H, \g)$ and by  $\dM^\huaB$ the group cohomology differential. Parallel to Definition
\ref{deficorbal}, we have
\begin{defi}
Let $\huaB$ be a relative Rota-Baxter operator on $(G, e_G, \cdot_G)$ with respect to an action $(H; \Phi)$. The cohomology of the cochain complex  $(C^*(\huaB)=\oplus_{k=1}^{+\infty}C^k(\huaB), \dM^\huaB)$  is defined to be  the {\bf  cohomology   for the relative Rota-Baxter operator $\huaB$}.
\end{defi}

Denote by $\huaH^{k}({\huaB})$ the $k$-th cohomology group.

\begin{rmk}
Our cohomology theory in $\RBgh$ is well rooted
via the theory of controlling algebras and is tested and justified
via the deformation theory. The above cohomology theory in $\RBGH$ on the level of
Lie groups will be justified via the Van-Est theory, which is the content of
the next section.
\end{rmk}

The following statements are important for the functoriality of our
cohomology theory.
\begin{pro}\label{HomomH}
Let $\huaB$ and $\huaB'$ be relative Rota-Baxter operators on a Lie group $G$ with respect to an action $(H;\Phi)$ and $(\Psi_G, \Psi_H)$ be a homomorphism from $\huaB'$ to $\huaB$. Then $\Psi_H$ is a Lie group homomorphism from
the descendent Lie group $(H, e_H, \star_{\huaB'})$ of $\huaB'$ to the descendent Lie group $(H, e_H, \star_{\huaB})$ of $\huaB$.
\end{pro}
\begin{proof}
By  \eqref{hom-rbo-gp1}, \eqref{hom-rbo-gp2} and the fact that $\Psi_H$ is a Lie group homomorphism, we have
\begin{eqnarray*}
\Psi_H(h_1\star_{\huaB'} h_2)&=&\Psi_H(h_1\cdot_H\huaB'(\Phi(h_1))h_2)=\Psi_H(h_1)\cdot_H\Psi(\huaB'(\Phi(h_1))h_2)\\
&=&\Psi_H(h_1)\cdot_H\Phi\Big(\huaB(\Psi_H(h_1))\Big)\Psi_H(h_2)=\Psi_H(h_1)\star_{\huaB}\Psi_H(h_2),
\end{eqnarray*}
which implies that $\Psi_H$ is a homomorphism between the descendent Lie groups.
\end{proof}

\begin{pro}\label{pro-homactgro}
Let $\huaB$ and $\huaB'$ be relative Rota-Baxter operators on $G$ with respect to an action $(H;\Phi)$ and $(\Psi_G, \Psi_H)$ be a homomorphism from $\huaB'$ to $\huaB$. Then the induced action $\Theta'$ of the Lie group $(H, \star_{\huaB'})$ on $G$ and the induced action $\Theta$ of the Lie group $(H, \star_{\huaB})$ on $G$ satisfy the following relation:
$$
\Psi_G\circ\Theta'(h)=\Theta(\Psi_H(h))\circ\Psi_G,\quad \forall h\in H.
$$
That is for all $h\in H$, the following diagram commutes:
$$
\xymatrix{
  G \ar[d]_{\Theta'(h)} \ar[r]^{\Psi_G}
                & G \ar[d]^{\Theta(\Psi_{H}(h))}  \\
    G \ar[r]_{\Psi_G}
                & G            .}
$$
\end{pro}
\begin{proof}
 By \eqref{hom-rbo-gp1}, \eqref{hom-rbo-gp2}, \eqref{eqactionH} and Proposition  \ref{HomomH}, for all $g\in G, h\in H$, we have
\begin{eqnarray*}
\Psi_G(\Theta'(h)g)&=&\Psi_G((\huaB'(\Phi(g)h^{\dag'}))^{-1}\cdot_{G}g\cdot_{G}\huaB'(h^{\dag'}))\\
&=&\Psi_G((\huaB'(\Phi(g)h^{\dag'}))^{-1})\cdot_{G}\Psi_G(g)\cdot_{G}\Psi_G(\huaB'(h^{\dag'}))\\
&=&\Psi_G(\huaB'(\Phi(g)h^{\dag'}))^{-1}\cdot_{G}\Psi_G(g)\cdot_{G}\huaB(\Psi_H(h^{\dag'}))\\
&=&\huaB(\Psi_H(\Phi(g)h^{\dag}))^{-1}\cdot_{G}\Psi_G(g)\cdot_{G}\huaB((\Psi_H(h))^{\dag})\\
&=&\huaB(\Phi(\Psi_G(g))(\Psi_H(h))^{\dag})^{-1}\cdot_G\Psi_G(g)\cdot_G\huaB((\Psi_H(h))^{\dag})\\
&=&\Theta(\Psi_{H}(h))\Psi_{G}(g).
\end{eqnarray*}
We finish the proof.
\end{proof}

Taking differentiation, we get
\begin{cor}\label{prohomactgro}
Let   $\huaB$ and $\huaB'$ be relative Rota-Baxter operators on $G$ with respect to an action $(H;\Phi)$, and $(\Psi_G, \Psi_H)$ be a homomorphism from $\huaB'$ to $\huaB$.  Then the action $\Theta'$ of the descendent Lie group $(H, e_H, \star_{\huaB'})$ on $\g$ and the action $\Theta$ of the descendent Lie group $(H, e_H, \star_{\huaB})$ on $\g$ satisfy the following relation:
\begin{equation*}
\Theta(\Psi_H(h))\circ(\Psi_G)_{*e_G}=(\Psi_G)_{*e_G}\circ\Theta'(h),\quad \forall h\in H.
\end{equation*}
\end{cor}

Now we
prove the functoriality of this cohomology theory.  Let  $\huaB$ and $\huaB'$ be relative Rota-Baxter operators on $G$ with respect to an action $(H;\Phi)$. Let $(\Psi_G, \Psi_H)$ be a homomorphism from $\huaB'$ to $\huaB$ in which $\Psi_H$ is invertible. Define a map $P: C^{k}(\huaB')\lon C^{k}(\huaB)$ by
\begin{equation*}
P(\omega)(h_1,\cdots,h_{k-1})=(\Psi_{G})_{*e_G}\omega(\Psi_H^{-1}(h_1),\cdots,\Psi_H^{-1}(h_{k-1})), \quad \forall h_i\in H.
\end{equation*}

\begin{thm}
With the above notations, $P$ is cochain map from the cochain complex $(C^*(\huaB'), \dM^{\huaB'})$ to $(C^*(\huaB), \dM^\huaB)$, which induces  a homomorphism $P_*$ from  the cohomology group $\huaH^{k}(\huaB')$   to $\huaH^{k}(\huaB)$ for all $k\geq 1$.
\end{thm}
\begin{proof}
 For all $\omega\in C^{k}(\huaB')$, by Corollary \ref{prohomactgro}, we have
\begin{eqnarray*}
&&\dM^{\huaB}(P\omega)(h_1,h_2,\cdots,h_{k})\\
&=&\Theta(h_1)(\Psi_{G})_{*e_G}\omega(\Psi_H^{-1}(h_2),\cdots,\Psi_H^{-1}(h_{k}))\\
&&+\sum_{i=1}^{k-1}(-1)^{i}P(\omega)(h_1,\cdots,h_{i-1}, h_i\star_{\huaB}h_{i+1}, h_{i+2}, \cdots, h_{k})\\
&&+(-1)^{k}P(\omega)(h_1,\cdots,h_{k-1})\\
&=&(\Psi_{G})_{*e_G}\Theta'(\Psi_H^{-1}(h_1))\omega(\Psi_H^{-1}(h_2),\cdots,\Psi_H^{-1}(h_{k}))\\
&&+\sum_{i=1}^{k-1}(-1)^{i}(\Psi_{G})_{*e_G}\omega(\Psi_H^{-1}(h_1),\cdots,\Psi_H^{-1}(h_{i-1}), \Psi_H^{-1}(h_i)\star_{\huaB'}\Psi_H^{-1}(h_{i+1}), \Psi_H^{-1}(h_{i+2}), \cdots, \Psi_H^{-1}(h_{k}))\\
&&+(-1)^{k}(\Psi_{G})_{*e_G}\omega(\Psi_H^{-1}(h_1),\cdots,\Psi_H^{-1}(h_{k-1}))\\
&=&(\Psi_{G})_{*e_G}\dM^{\huaB'}\omega(\Psi_H^{-1}(h_1),\cdots,\Psi_H^{-1}(h_{k}))\\
&=&P(\dM^{\huaB'}\omega)(h_1,\cdots,h_{k}),\quad \forall h_i\in H.
\end{eqnarray*}
Thus $P$ is a cochain map. Consequently it  induces a homomorphism $P_*$ from   the cohomology group $\huaH^{k}(B')$  to  $\huaH^{k}(B)$ for all $k\geq 1$.
\end{proof}

\section{Differentiation and Van Est theory}\label{sec:ve}

In this section, we establish the differentiation functor for  relative
Rota-Baxter operators and the Van Est homomorphism on  the cohomology
level.  The classical Van Est isomorphism \cite{Van}  gives the
relation between the smooth cohomology of Lie groups and the
cohomology of Lie algebras. See \cite{AC,CD,Cr, Hou, Li,MS,Pfl,WX}
for various Van Est type theorems.

\subsection{The differentiation functor for  relative
Rota-Baxter operators}

Let $\Phi$ be an action of a Lie group $(G, e_G, \cdot_G)$ on a manifold $M$. Define    $\Delta: G\times M\rightarrow M$ by   $\Delta(g, m)=\Phi(g)m$. Then
 $\Delta(e_{G}, m)=m$ and $\Delta(g_1\cdot_G g_2, m)=\Delta(g_{1}, \Delta(g_2, m))$, for all $g_1, g_2\in G, m\in M$. Moreover,  we have the following Leibniz rule,
\begin{eqnarray*}
\frac{d}{dt}\bigg|_{t=0}\Delta(\gamma_1(t), \gamma_2(t))&=&\frac{d}{dt}\bigg|_{t=0}\Delta(\gamma_1(0), \gamma_2(t))+\frac{d}{dt}\bigg|_{t=0}\Delta(\gamma_1(t), \gamma_2(0))\\
&=&\frac{d}{dt}\bigg|_{t=0}\Phi(\gamma_1(t))\gamma_2(t)\\
&=&\frac{d}{dt}\bigg|_{t=0}\Phi(\gamma_1(0))\gamma_2(t)+\frac{d}{dt}\bigg|_{t=0}\Phi(\gamma_1(t))\gamma_2(0),
\end{eqnarray*}
where $\gamma_1(t)$ is a path in $G$ with $\gamma_1(0)=g$ and $\gamma_2(t)$ is a path in $M$ with $\gamma_{2}(0)=m$.

Let $(G, e_G, \cdot_G)$ and $(H, e_H, \cdot_H)$ be Lie groups and denote by $\exp_G$ and $\exp_H$ the exponential maps for the Lie groups $(G, e_G, \cdot_G)$ and $(H, e_H, \cdot_H)$ respectively.
Let $\Phi: G\rightarrow \Aut(H)$ be an action of $G$ on $H$. Since
$\Phi(g)\in\Aut(H)$ for all $g\in G$, then $\Phi(g)_{*e_H}:
\h\rightarrow\h$ is a Lie algebra isomorphism. By $\Phi(g_1\cdot_G
g_2)=\Phi(\g_1)\Phi(g_2)$, we have $\Phi(g_1\cdot_G
g_2)_{*e_H}=\Phi(g_1)_{*e_H}\Phi(g_2)_{*e_H}$. Thus we obtain a Lie
group homomorphism from $G$ to $\Aut(\h)$, which we denote by
$\tilde{\Phi}:G\lon\Aut(\h)$.
Then taking the differential, we obtain a Lie algebra homomorphism
$\phi:=\tilde{\Phi}_{*e_G}$ from the Lie algebra $\g$ to $\Der(\h)$. We call
$\phi$ the differentiated action of $\Phi$. In fact, Lie II theorem tells us that
$\Aut(H)\cong \Aut(\h)$, if $H$ is connected and simply
connected. Therefore when both $G$ and $H$
are connected and simply connected, given a Lie algebra action $\phi: \g \to
\Der(\h)$, there is a unique integrated action $\Phi: G\to \Aut(H)$
whose differentiation is $\phi$.  This procedure can be well explained by the following diagram:
\begin{equation}\label{eq:relation}
\small{ \xymatrix{G\ar@{=}[d]  \ar[rr]^{\Phi} & & \Aut(H) \ar[d]^{\text{differentiation}}\\
 G \ar[rr]^{\tilde{\Phi}}\ar[d]^{\text{differentiation}} &  &  \Aut(\h) \ar[d]^{\text{differentiation}}  \\
 \g \ar[rr]^{  \phi=\tilde{\Phi}_{*e_G} } & &\Der(\h).}
}
\end{equation}

\begin{thm}\label{thmhuaBb}
Let  $\huaB:H\lon G$ be a relative Rota-Baxter operator on $G$ with respect to an action $(H;\Phi)$. Define $B:\h\rightarrow\g$ by
\begin{equation}\label{eq:BB}
B=\huaB_{*e_{H}},
\end{equation}
which is the tangent map of $\huaB$ at the identity $e_{H}$. Then $B$
is a relative Rota-Baxter operator on $\g$ with respect to the
action $(\h;\phi)$, where $\phi$ is the differentiated action
of $\Phi$ defined in \eqref{eq:relation}.
\end{thm}
\begin{proof}
 For all $u, v\in\h$, we have
\begin{eqnarray*}
&&[B(u), B(v)]_{\g}\\
&=&\frac{d^2}{dtds}\bigg|_{t,s=0}\exp_G(tB(u))\cdot_G\exp_G(sB(v))\cdot_G\exp_G(-tB(u))\\
&=&\frac{d^2}{dtds}\bigg|_{t,s=0}\huaB(\exp_H(tu))\cdot_G\huaB(\exp_H(sv))\cdot_G(\huaB(\exp_H(tu)))^{-1}\\
&=&\frac{d^2}{dtds}\bigg|_{t,s=0}\huaB(\exp_H(tu))\cdot_G\huaB(\exp_H(sv))\cdot_G\huaB\Big(\Phi(\huaB(\exp_H(tu))^{-1})\exp_H(-tu)\Big)\\
&=&\frac{d^2}{dtds}\bigg|_{t,s=0}\huaB\Big(\exp_H(tu)\cdot_H\Phi(\huaB(\exp_H(tu)))\exp_H(sv)\Big)\cdot_G\huaB\Big(\Phi(\huaB(\exp_H(tu))^{-1})\exp_H(-tu)\Big)\\
&=&\frac{d^2}{dtds}\bigg|_{t,s=0}\huaB\Big(\Ad_{\exp_H(tu)}(\Phi(\huaB(\exp_H(tu)))\exp_H(sv))\cdot_H(\exp_H(tu)\cdot_H\Phi(\Ad_{\huaB(\exp_H(tu))}\huaB(\exp_H(sv)))\exp_H(-tu))\Big)\\
&=&\huaB_{*e_G}\frac{d^2}{dtds}\bigg|_{t,s=0}\Ad_{\exp_H(tu)}(\Phi(\huaB(\exp_H(tu)))\exp_H(sv))\cdot_H(\exp_H(tu)\cdot_H\Phi(\Ad_{\huaB(\exp_H(tu))}\huaB(\exp_H(sv)))\exp_H(-tu))\\
&=&\huaB_{*e_G}\frac{d^2}{dtds}\bigg|_{t,s=0}\Big(\Ad_{\exp_H(tu)}\Phi(\huaB(\exp_H(tu)))\exp_H(sv)\Big)\\
&&+\huaB_{*e_G}\frac{d^2}{dtds}\bigg|_{t,s=0}\exp_H(tu)\cdot_H\Phi(\Ad_{\huaB(\exp_H(tu))}\huaB(\exp_H(sv)))\exp_H(-tu)\\
&=&\huaB_{*e_G}\Big(\frac{d^2}{dtds}\bigg|_{t,s=0}\Ad_{\exp_H(tu)}\exp_H(sv)+ \frac{d^2}{dtds}\bigg|_{t,s=0}\Phi(\huaB(\exp_H(tu)))\exp_H(sv)\Big)\\
&&+\huaB_{*e_G}\frac{d^2}{dsdt}\bigg|_{t,s=0}\exp_H(tu)\cdot_H\Phi(\Ad_{\huaB(\exp_H(tu))}\huaB(\exp_H(sv)))\exp_H(-tu)\\
&=&\huaB_{*e_G}\Big([u,v]_\h+ \phi(\huaB_{*e_G}(u))v\Big)+\huaB_{*e_G}\frac{d^2}{dsdt}\bigg|_{t,s=0}\Phi(\Ad_{\huaB(\exp_H(tu))}\huaB(\exp_H(sv)))\exp_H(-tu)\\
&=&\huaB_{*e_G}\Big([u,v]_\h+\phi(\huaB_{*e_G}(u))v\Big)+\huaB_{*e_G}\frac{d^2}{dsdt}\bigg|_{t,s=0}\Phi(\huaB(\exp_H(sv)))\exp_H(-tu)\\
&=&\huaB_{*e_G}\Big([u,v]_\h+ \phi(\huaB_{*e_G}(u))v\Big)-\huaB_{*e_G}(\phi(\huaB_{*e_G}(v))u)\\
&=&B([u,v]_\h+\phi(B(u))v-\phi(B(v))u).
\end{eqnarray*}
Thus $B$ is a relative Rota-Baxter operator.
\end{proof}

The above result extends to the level of morphisms:

\begin{pro}\label{pro:diff-functor}
Let $\huaB$ and $\huaB'$ be relative Rota-Baxter operators on a Lie group $G$ with respect to an action $(H;\Phi)$, and $(\Psi_G, \Psi_H)$ be a homomorphism from $\huaB'$ to $\huaB$. Then $((\Psi_{G})_{*e_G}, (\Psi_{H})_{*e_H})$ is a homomorphism from the relative Rota-Baxter operator $\huaB'_{*e_H}$ to $\huaB_{*e_H}$.
\end{pro}
\begin{proof}
We denote  $((\Psi_{G})_{*e_G}, (\Psi_{H})_{*e_H})$ by $(\psi_\g, \psi_\h)$ and $\huaB'_{*e_H}=B', \huaB_{*e_H}=B$. Since $\Psi_G: G\lon G$ and $\Psi_H: H\lon H$ are Lie group homomorphisms, thus $\psi_\g: \g\lon\g$ and $\psi_\h: \h\lon\h$ are Lie algebra homomorphisms. By \eqref{hom-rbo-gp1}, we have $B\circ\psi_{\h}=\psi_{\g}\circ B'$.  Then we have
\begin{eqnarray*}
\frac{d}{ds}\frac{d}{dt}\bigg|_{s,t=0}\Psi_{H}(\Phi(\exp_G(sx))\exp_H(tu))
=\psi_{\h}\frac{d}{ds}\frac{d}{dt}\bigg|_{t,s=0}\Phi(\exp_G(sx))\exp_H(tu)=\psi_{\h}(\phi(x)u)
\end{eqnarray*}
and
\begin{eqnarray*}
\frac{d}{ds}\frac{d}{dt}\bigg|_{t,s=0}\Phi(\Psi_G(\exp_{G}(sx)))\Psi_H(\exp_H(tu))
&=&\frac{d}{ds}\frac{d}{dt}\bigg|_{t,s=0}\Phi(\exp_{G}(s\psi_{\g}(x)))\exp_{H}(t\psi_{\h}(u))\\
&=&\phi(\psi_{\g}(x))\psi_{\h}(u), \quad \forall x\in\g, u\in\h.
\end{eqnarray*}
By \eqref{hom-rbo-gp2}, we have $\psi_{\h}(\phi(x)u)=\phi(\psi_{\g}(x))\psi_{\h}(u)$, which implies that $((\Psi_{G})_{*e_G}, (\Phi_{H})_{*e_H})$ is a homomorphism from $\huaB'_{*e_H}$ to $\huaB_{*e_H}$.
\end{proof}

Given Lie groups $G, H$ and their Lie algebras $\g, \h$ respectively,
then the above differentiation procedure gives us a functor \begin{equation}\label{eq:diff}\Diff:
\RBGH \to \RBgh. \end{equation}

Actually the set of homomorphisms between relative Rota-Baxter
operators on Lie groups and the set of homomorphisms between the
differentiated relative Rota-Baxter operators on the corresponding Lie
algebras are isomorphic. First we need a lemma:

\begin{lem}\label{pro-imporpair}
Let $G$ and $H$ be connected Lie groups whose Lie algebras are  $\g$ and $\h$ respectively. Let $\Phi: G\rightarrow\Aut(H)$ be an action
  $G$ on  $H$ and $\phi:\g\lon\Der(\h)$ be the induced  action of $\g$ on $\h$. Let  $\Psi_G:G\lon G$ and $\Psi_H: H\lon H$ be Lie groups homomorphisms and $\psi_\g:\g\lon\g$ and $\psi_\h:\h\lon\h$ be the induced Lie algebra homomorphisms. If $\psi_\g$ and $\psi_\h$  satisfies \eqref{hom-rbo2}.
Then $\Psi_G$ and $\Psi_H$ satisfy  \eqref{hom-rbo-gp2}, i.e.
\begin{equation*}
\Phi(\Psi_G(g))\Psi_H(h)=\Psi_H(\Phi(g)h), \quad \forall g\in G, h\in H.
\end{equation*}
\end{lem}
\begin{proof}
We first show that  \eqref{hom-rbo-gp2} holds for $g$ and $h$ in a small neighborhood
 of identities. In fact,  by writing $g=\exp_G x$ and $h=\exp_H u$   and using
 \eqref{hom-rbo2} and Diagram \eqref{eq:relation}, we have
\begin{eqnarray*}
\Psi_H(\Phi(\exp_G(x))\exp_H(u))&=&\Psi_H(\exp_H(\tilde{\Phi}(\exp_G(x))u))=\Psi_H(\exp_H(\exp_{\Aut(\h)}(\phi(x))u))\\
&=&\exp_H(\psi_\h(\exp_{\Aut(\h)}(\phi(x))u))=\exp_H(\exp_{\Aut(\h)}(\phi(\psi_\g(x)))\psi_\h(u))\\
&=&\exp_H(\tilde{\Phi}(\exp_G(\psi_\g(x)))\psi_\h(u))=\Phi(\Psi_G(\exp_G(x)))\exp_H(\psi_\h(u))\\
&=&\Phi(\Psi_G(\exp_G(x)))\Psi_H(\exp_H(u)).
\end{eqnarray*} Since $G$ and $H$ are connected, we can
  write
$G=\cup_{n\in\mathbb{N}}U_{G}^{n}$ and
$H=\cup_{n\in\mathbb{N}}U_{H}^{n}$,  for some small open set $U_G\subseteq G$ and $U_H\subseteq H$
containing $e_G$ and $e_H$ respectively.  Then $\Psi_H$ being a Lie group homomorphism
 and $\Phi$ being an action of Lie groups imply that for $g=\exp_G x$, and $h=h_1
 \cdot \dots \cdot h_k$, but each $h_j$ very small,
 \eqref{hom-rbo-gp2} still holds; finally, the same facts imply that for
 $g=g_1 \cdot \dots \cdot g_l$, with each $g_j$ very small, and any
 $h\in H$,
  \eqref{hom-rbo-gp2} still holds. Thus we
 have shown that  \eqref{hom-rbo-gp2} holds in general.
\end{proof}

\begin{thm}
Let $G$ and $H$ be connected, simply connected Lie groups, $\huaB'$ and $\huaB$ be relative Rota-Baxter operators on $G$ with respect to the action $(H; \Phi)$. Then $$\Hom_\RBGH(\huaB',\huaB)\cong\Hom_\RBgh(\Diff(\huaB'),\Diff(\huaB)).$$
\end{thm}
\begin{proof}
For any $(\Psi_G, \Psi_H)\in\Hom_\RBGH(\huaB',\huaB)$, we obtain $((\Psi_G)_{*e_G},(\Psi_H)_{*e_H})\in\Hom_\RBgh(\Diff(\huaB'),\Diff(\huaB))$ via Proposition \ref{pro:diff-functor}.

Conversely, for any
$(\psi_\g,\psi_\h)\in\Hom_\RBgh(\Diff(\huaB'),\Diff(\huaB))$,
$(\psi_\g,\psi_\h)$ is a Lie algebra endomorphism of the semi-direct
product Lie algebra $\g\ltimes_{\phi}\h$ and
$(\psi_\g,\psi_\h)(Gr(B'))\subseteq Gr(B)$. By Lemma
\ref{pro-imporpair}, there exist  unique Lie group homomorphisms
$\Psi_G:G\lon G$ and $ \Psi_H: H\lon H$ such that $(\Psi_G, \Psi_H)$
is a Lie group endomorphism of the semi-direct product  Lie group
$G\ltimes_\Phi H$. Since the image $(\Psi_G,\Psi_H)Gr(\huaB')$ is a Lie subgroup
of $G\ltimes_\Phi H$ and its Lie algebra is $(\psi_\g,\psi_\h)(Gr(B'))
\subseteq Gr(B)$, it follows\footnote{It is obvious that the inclusion
happens locally near identity, the global statement follows again
using the fact that both $(\Psi_G,\Psi_H)Gr(\huaB')$ and $Gr(\huaB)$ are
connected and they can be written as products of elements near
identity. } that $(\Psi_G,\Psi_H)Gr(\huaB')\subseteq
Gr(\huaB)$.  This in turn implies that $\Psi_G\circ \huaB'=\huaB\circ \Psi_H$. Thus $(\Psi_G,\Psi_H)\in \Hom_\RBGH(\huaB',\huaB)$, and
$\Hom_\RBGH(\huaB',\huaB)\cong\Hom_\RBgh(\Diff(\huaB'),\Diff(\huaB)).$
\end{proof}

\subsection{The Van Est homomorphism for the cohomologies}
 Now we extend the above differentiation  to the level of cohomology. For
this, let us consider the relation between the descendent Lie group of the relative Rota-Baxter operator $\huaB$ and the descendent Lie algebra of the  relative Rota-Baxter operator $\Diff(\huaB)$.

\begin{pro}\label{pro:descendent-diff}
Let  $\huaB$ be a relative Rota-Baxter operator on $G$ with respect to
the action $(H;\Phi)$.   Then the Lie algebra
$(\h,[\cdot,\cdot]_\star)$ of the descendent Lie group $(H, e_H,
\star)$ is the descendent Lie algebra $(\h,
[\cdot,\cdot]_{\Diff(\huaB)})$ of the relative Rota-Baxter operator
$\Diff(\huaB)$ on the level of Lie algebras.
\emptycomment{ \begin{equation*}
[u, v]_B=\phi(B(u))v-\phi(B(v))u+[u,v]_\h, \quad\forall u, v\in\h.
\end{equation*}
This proposition can be well explained by the following diagram:
 $$
\xymatrix@!0{
  & G
      &  & H \ar[ll]_{\huaB}        \\
  G \ar@{=}[ru]
      &  & H \ar@{=}[ru]\ar[ll]_{\quad\huaB} \\
  & \g\ar@{.>}[uu]^{\exp} \ar@{=}[ld]
      &  & \h\ar@{.>}[ll]_{ B\quad}\ar[uu]_{\vspace{-1cm}\exp} \ar@{=}[ld]               \\
  \g \ar[uu]^{\exp}
      &  & \h .\ar[uu]_{\Exp} \ar[ll]_{B}      }
$$
}
\end{pro}

\begin{proof}
We denote  by $\Exp$ the exponential map for the Lie group $(H, e_H,
\star)$. For all $u, v\in\h$, with the help of Proposition \ref{pro:dec-gp}, we have
\begin{eqnarray*}
&&[u, v]_\star\\
&=&\frac{d^2}{dtds}\bigg|_{t,s=0}\Exp(tu)\star\Exp(sv)\star\Exp(-tu)\\
&=&\frac{d^2}{dtds}\bigg|_{t,s=0}\Big(\Exp(tu))\cdot_H\Phi(\huaB(\Exp(sv)))\Exp(sv)\Big)\star\Exp(-tu)\\
&=&\frac{d^2}{dtds}\bigg|_{t,s=0}\Big(\Exp(tu))\cdot_H\Phi(\huaB(\Exp(sv)))\Exp(sv)\Big)\star\Phi(\huaB(\Exp(tu))^{-1})(\Exp(tu))^{-1}\\
&=&\frac{d^2}{dtds}\bigg|_{t,s=0}\Big(\Exp(tu)\cdot_H\Phi(\huaB(\Exp(tu)))\Exp(sv)\Big)\cdot_H\Phi(\Ad_{\huaB(\Exp(tu))}\huaB(\Exp(sv)))(\Exp(tu))^{-1}\\
&=&\frac{d^2}{dtds}\bigg|_{t,s=0}\Big(\Ad_{\Exp(tu)}(\Phi(\huaB(\Exp(tu)))\Exp(sv))\Big)\cdot_H\Big(\Exp(tu)\cdot_H\Phi(\Ad_{\huaB(\Exp(tu))}\huaB(\Exp(sv)))(\Exp(tu))^{-1}\Big)\\
&=&\frac{d^2}{dtds}\bigg|_{t,s=0}\Big(\Ad_{\Exp(tu)}\Phi(\huaB(\Exp(tu)))\Exp(sv)\Big)\\
&&+\frac{d^2}{dtds}\bigg|_{t,s=0}\Exp(tu)\cdot_H\Phi(\Ad_{\huaB(\Exp(tu))}\huaB(\Exp(sv)))(\Exp(tu))^{-1}\\
&=&\frac{d^2}{dtds}\bigg|_{t,s=0}\Ad_{\Exp(tu)}\Exp(sv)+\frac{d^2}{dtds}\bigg|_{t,s=0}\Phi(\huaB(\Exp(tu)))\Exp(sv)\\
&&+\frac{d^2}{dsdt}\bigg|_{t,s=0}\Exp(tu)\cdot_H\Phi(\Ad_{\huaB(\Exp(tu))}\huaB(\Exp(sv)))(\Exp(tu))^{-1}\\
&=&[u,v]_\h+\phi(\huaB_{*e_H}(u))v+\frac{d^2}{dsdt}\bigg|_{t,s=0}\Phi(\Ad_{\huaB(\Exp(tu))}\huaB(\Exp(sv)))(\Exp(tu))^{-1}\\
&=&[u,v]_\h+\phi(\huaB_{*e_H}(u))v+\frac{d^2}{dsdt}\bigg|_{t,s=0}\Phi(\huaB(\Exp(sv)))(\Exp(tu))^{-1}\\
&=&[u,v]_\h+\phi(\huaB_{*e_H}(u))v-\phi(\huaB_{*e_H}(v))u
\end{eqnarray*}
Therefore, $[\cdot,\cdot]_\star=[\cdot,\cdot]_{\Diff(\huaB)}$, which implies that the Lie algebra  of the descendent Lie group $(H, e_H, \star)$ is the descendent Lie algebra   of the relative Rota-Baxter operator $\Diff(\huaB) $.
\end{proof}

By Theorem \ref{groupact}, we know that $\Theta$ is an action of the descendent Lie group $(H, e_H, \star)$ on $G$. By Lemma \ref{lem:rep-on-g}, $\Theta:H\lon GL(\g)$ is an action of the descendent Lie group $(H, e_H, \star)$ on the vector space $\g$. Taking the differentiation, we have the following result.

\begin{pro}\label{pro:rep-diff}
Let  $\huaB:H\lon G$ be a relative Rota-Baxter operator on $G$ with respect to an action $(H;\Phi)$.   Then the differentiation of the action $\Theta:H\lon GL(\g)$ of the descendent Lie group $(H, e_H, \star)$ on $\g$ is exactly the representation  $\theta: \h\lon\gl(\g)$  of the descendent Lie  algebra $(\h,[\cdot,\cdot]_{\Diff(\huaB) })$ on $\g$ given in Proposition \ref{pro:rep-on-g-alg}. That is
\begin{equation*}
\Theta_{*e_H}(u)x=\Diff(\huaB) (\phi(x)u)+[\Diff(\huaB) (u), x]_\g, \quad \forall x\in\g, u\in\h.
\end{equation*}
\end{pro}

\begin{proof}
By the definition of $\Theta$, and use the above notations in
Proposition \ref{pro:diff-functor} and Proposition \ref{pro:descendent-diff},  we have
\begin{eqnarray*}
\Theta_{*e_H}(u)x
&=&\frac{d}{dt}\frac{d}{ds}\bigg|_{t,s=0}\Theta(\Exp(tu))\exp_G(sx)\\
&=&\frac{d}{dt}\frac{d}{ds}\bigg|_{t,s=0}\huaB(\Phi(\exp_G(sx))\Exp(-tu))^{-1}\cdot_{G}\exp_G(sx)\cdot_{G}\huaB(\Exp(-tu))\\
&=&\frac{d}{dt}\frac{d}{ds}\bigg|_{t,s=0}\huaB(\Exp(-tu))^{-1}\cdot_{G}\exp_G(sx)\cdot_{G}\huaB(\Exp(-tu))\\
&&+\frac{d}{dt}\frac{d}{ds}\bigg|_{t,s=0}\huaB(\Phi(\exp_G(sx))\Exp(-tu))^{-1}\cdot_{G}\huaB(\Exp(-tu))\\
&=&[\Diff(\huaB) (u),x]_\g+\frac{d}{ds}\frac{d}{dt}\bigg|_{t,s=0}\huaB(\Phi(\exp_G(sx))\Exp(-tu))^{-1}\\
&=&[\Diff(\huaB) (u),x]_\g+\Diff(\huaB) (\phi(x)u),
\end{eqnarray*}
which implies that $\Theta_{*e_H}=\theta$.
\end{proof}

Let  $\huaB:H\lon G$ be a relative Rota-Baxter operator on $G$ with
respect to an action $(H;\Phi)$, we define $\VE:C^k(\huaB)\lon C^{k}(\Diff(\huaB))$ by
\begin{eqnarray}\label{defieqVE}
&&\VE (F)(u_1,\cdots, u_{k-1})\\
\nonumber&=& \sum_{s\in \Sym(k-1)}(-1)^{|s|}\frac{d}{dt_{s(1)}}\cdots\frac{d}{dt_{s(k-1)}}\bigg|_{t_{s(1)}=t_{s(2)}=\cdots=t_{s(k-1)}=0}F(\Exp(t_{s(1)}u_{s(1)}),\cdots,\Exp(t_{s(k-1)}u_{s(k-1)})),
\end{eqnarray}
for all $F\in C^k(\huaB),~u_1,\cdots, u_{k-1}\in\h$.

\begin{thm}\label{thm:van-est1}
  With the above notations, $\VE$ is a cochain map, i.e. we have the following commutative diagram
\[
\small{ \xymatrix{
\cdots\longrightarrow C^{k-1}(\huaB) \ar[d]^{\VE} \ar[r]^{\quad\dM^{\huaB}}&
C^{k}(\huaB) \ar[d]^{\VE} \ar[r]^{\quad\dM^{\huaB}} & C^{k+1}(\huaB) \ar[d]^{\VE} \ar[r]  & \cdots  \\
\cdots\longrightarrow C^{k-1}(\Diff(\huaB) ) \ar[r]^{\quad\dM_{\CE}^{\Diff(\huaB)}} & C^{k}(\Diff(\huaB)) \ar[r]^{\quad \dM_{\CE}^{\Diff(\huaB)}} &C^{k+1}(\Diff(\huaB))\ar[r]& \cdots.}
}
\]
Consequently, $\VE$ induces a homomorphism $\VE_*$ from the cohomology
group $\huaH^{k}(\huaB)$ to $H^{k}(\Diff(\huaB))$. The map $\VE$ is called the {\bf Van Est map}.
\end{thm}
\begin{proof}
By Proposition \ref{pro:descendent-diff}, the differentiation of the
descendent Lie group $(H, e_H, \star)$ is the descendent Lie algebra
$(\h, [\cdot,\cdot]_{\Diff(\huaB)})$. By Proposition
\ref{pro:rep-diff}, the differentiation of the action $\Theta$
of the descendent Lie group $(H, e_H, \star)$ is the representation
$\theta$ of the descendent Lie algebra $(\h,
[\cdot,\cdot]_{\Diff(\huaB)})$. Since the cochains $C^k(\huaB)$ and
$C^k(\Diff(\huaB))$ are in fact exactly those of  the descendent Lie
group $(H, e_H, \star)$ and the  descendent Lie algebra $(\h,
[\cdot,\cdot]_{\Diff(\huaB)})$ respectively, the conclusion follows
from the classical argument for the cohomologies of Lie groups and Lie
algebras.
\end{proof}

Just as the classical situation, under certain conditions, the cohomology group $\huaH^{k}(\huaB)$ and $H^{k}(\Diff(\huaB))$ are isomorphic.

\begin{thm}
If the Lie group $(H, e_H, \cdot_H)$ is connected and its homotopy
groups are trivial in $1, 2,\cdots, n$, then for $1\leq k\leq n$, the
cohomology group $\huaH^{k}(\huaB)$ of the relative Rota-Baxter
operator $\huaB$ on the Lie group $G$ is isomorphic to the cohomology
group $H^{k}(\Diff(\huaB))$ of the relative Rota-Baxter operator on
the level of Lie algebras.
\end{thm}
\begin{proof}
As for the previous theorem, the conclusion also follows from the classical argument for the cohomologies of Lie groups and Lie algebras.
\end{proof}

\subsection{An explicit example}
At the end of this section, we give a concrete example to demonstrate the
differentiation procedure, and we also calculate a second cohomology
group in this example.

\begin{ex}{\rm
Consider the Euclidean Lie group $$G=\{X|X=\left(
                                             \begin{array}{cc}
                                               A & \alpha \\
                                               0 & 1 \\
                                             \end{array}
                                           \right), A\in\SO(n), \alpha\in \R^{n}
\}.$$ Then we have $G=G_{1}G_{2}$, where $G_{1}=\{\left(
                                                  \begin{array}{cc}
                                                    A & 0 \\
                                                    0 & 1 \\
                                                  \end{array}
                                                \right)|A\in\SO(n)\}$
and $G_{2}=\{\left(
             \begin{array}{cc}
               I_{n\times n} & \alpha \\
               0 & 1 \\
             \end{array}
           \right)|\alpha\in\R^{n}\}$.
           Since any $\left( \begin{array}{cc}
                                               A & \alpha \\
                                               0 & 1 \\
                                             \end{array}
                                           \right)  $  can be written as $\left( \begin{array}{cc}
                                               A & \alpha \\
                                               0 & 1 \\
                                             \end{array}
                                           \right)  =\left( \begin{array}{cc}
                                               A & 0\\
                                               0 & 1 \\
                                             \end{array}
                                           \right) \left( \begin{array}{cc}
                                               I_{n\times n} & A^T\alpha \\
                                               0 & 1 \\
                                             \end{array}
                                           \right),$
                                           by  Example \ref{lem:RBPG},  the map $\huaB: G\lon G$ defined by
           \begin{equation*}
           \huaB(\left( \begin{array}{cc}
                                                                             A & \alpha \\
                                                                             0 & 1 \\
                                                                           \end{array}
                                                                         \right)
           )=\left(
               \begin{array}{cc}
                 I_{n\times n} & -A^{T}\alpha \\
                 0 & 1 \\
               \end{array}
             \right)
           ,\quad \forall \left(
                             \begin{array}{cc}
                               A & \alpha \\
                               0 & 1 \\
                             \end{array}
                           \right)\in G
           \end{equation*}
  is a   Rota-Baxter operator on $G $. By Proposition \ref{pro:dec-gp}, the multiplication $\star$ in the descendent Lie group   $(G, \star)$ is given by
\begin{equation*}
\left(
  \begin{array}{cc}
    A & \alpha \\
    0 & 1 \\
  \end{array}
\right)\star\left(
              \begin{array}{cc}
                C & \beta \\
                0 & 1 \\
              \end{array}
            \right)=\left(
                      \begin{array}{cc}
                        AC & ACA^{T}\alpha+A\beta \\
                        0 & 1 \\
                      \end{array}
                    \right)
\end{equation*}
and the inverse is given by  $\left(
       \begin{array}{cc}
         A & \alpha \\
         0 & 1 \\
       \end{array}
     \right)^\dagger=\left(
                       \begin{array}{cc}
                         A^{T} & -(A^{T})^2\alpha \\
                         0 & 1 \\
                       \end{array}
                     \right)$.  It is straightforward to deduce that $G$ is the direct product of $G_1$ and $G_2$, i.e. $G=G_1\star G_2$, where $G_1=\{Y|Y=\left(
                                                                                                          \begin{array}{cc}
                                                                                                            C & 0 \\
                                                                                                            0 & 1 \\
                                                                                                          \end{array}
                                                                                                        \right), C\in \SO(n)\}$
and $G_2=\{Z|Z=\left(
                 \begin{array}{cc}
                   I & \beta \\
                   0 & 1 \\
                 \end{array}
               \right), \beta\in\R^{n}\}$.
                     By Theorem \ref{groupact}, the action $\Theta$ of the descendent Lie group $(G, \star)$ on the manifold $G$ is given by
\begin{eqnarray*}
&&\Theta(\left(
        \begin{array}{cc}
          A & \alpha \\
          0 & 1 \\
        \end{array}
      \right)
)\left(
         \begin{array}{cc}
           C & \beta \\
           0 & 1 \\
         \end{array}
       \right)\\
&=&\Big(\huaB(\left(
                \begin{array}{cc}
                  C & \beta \\
                  0 & 1 \\
                \end{array}
              \right)\left(
                       \begin{array}{cc}
                         A^{T} & -(A^{T})^2\alpha \\
                         0 & 1 \\
                       \end{array}
                     \right)\left(
                \begin{array}{cc}
                  C^{T} & -C^{T}\beta \\
                  0 & 1 \\
                \end{array}
              \right))
\Big)^{-1}\left(
            \begin{array}{cc}
              C & \beta \\
              0 & 1 \\
            \end{array}
          \right)\huaB(\left(
                       \begin{array}{cc}
                         A^{T} & -(A^{T})^2\alpha \\
                         0 & 1 \\
                       \end{array}
                     \right))\\
&=&\left(
     \begin{array}{cc}
       C & CAC^{T}\beta \\
       0 & 1 \\
     \end{array}
   \right),
\end{eqnarray*}
and the induced representation, which we use the same notation $\Theta: G\lon GL(\g)$ is given by
$$\Theta(\left(
              \begin{array}{cc}
                A & \alpha \\
                0 & 1 \\
              \end{array}
            \right)
)\left(
    \begin{array}{cc}
      x & u \\
      0 & 0 \\
    \end{array}
  \right)=\frac{d}{dt}\bigg|_{t=0}\Theta(\left(
              \begin{array}{cc}
                A & \alpha \\
                0 & 1 \\
              \end{array}
            \right))\left(
    \begin{array}{cc}
      \exp(tx) & tu \\
      0 & 1 \\
    \end{array}
  \right)=\left(
             \begin{array}{cc}
               x & Au \\
               0 & 0 \\
             \end{array}
           \right),
  $$
where  the Lie algebra $\g$ of the Euclidean Lie group $G$ is given by $$\g=\{\left(
                                                                        \begin{array}{cc}
                                                                          x & u \\
                                                                          0 & 0 \\
                                                                        \end{array}
                                                                      \right)|x\in\so(n), u\in\R^{n}\}.$$
Then the differentiation $\Diff(\huaB) $ is given by
 $$\Diff(\huaB) (\left(                                                   \begin{array}{cc}
                                                                                                     x & u \\
                                                                                                     0 & 0 \\
                                                                                                   \end{array}
                                                                                                 \right)
                                                                      )
                                                                      )=\frac{d}{dt}\bigg|_{t=0}\huaB(\left(
                                                                                             \begin{array}{cc}
                                                                                               \exp(tx) & tu \\
                                                                                               0 & 1 \\
                                                                                             \end{array}
                                                                                           \right))=\left(
                                                                                                      \begin{array}{cc}
                                                                                                        0 & -u \\
                                                                                                        0 & 0 \\
                                                                                                      \end{array}
                                                                                                    \right),
                                                                      $$ for all $\left(
                                                                                         \begin{array}{cc}
                                                                                           x & u \\
                                                                                           0 & 0 \\
                                                                                         \end{array}
                                                                                       \right)
                                                                      \in\g$.
By Proposition \ref{pro:descendent-diff}, the descendent Lie algebra  $(\g, [\cdot,\cdot]_{\Diff(\huaB)})$ is given by
\begin{eqnarray*}
&&[\left(
   \begin{array}{cc}
     x & u \\
     0 & 0 \\
   \end{array}
 \right), \left(
            \begin{array}{cc}
              y & v \\
              0 & 0 \\
            \end{array}
          \right)]_{\Diff(\huaB)}\\
 &=&\left(
      \begin{array}{cc}
        x & u \\
        0 & 0 \\
      \end{array}
    \right)\left(
             \begin{array}{cc}
               y & v \\
               0 & 0 \\
             \end{array}
           \right)-\left(
      \begin{array}{cc}
        y & v \\
        0 & 0 \\
      \end{array}
    \right)\left(
             \begin{array}{cc}
               x & u \\
               0 & 0 \\
             \end{array}
           \right)\\
           &&+[\Diff(\huaB) (\left(
             \begin{array}{cc}
               x & u \\
               0 & 0 \\
             \end{array}
           \right)),\left(
      \begin{array}{cc}
        y & v \\
        0 & 0 \\
      \end{array}
    \right)]-[\Diff(\huaB) (\left(
             \begin{array}{cc}
               y & v \\
               0 & 0 \\
             \end{array}
           \right)),\left(
      \begin{array}{cc}
        x & u \\
        0 & 0 \\
      \end{array}
    \right)]\\
 &=&\left(
      \begin{array}{cc}
        xy-yx & xv-yu \\
        0 & 0 \\
      \end{array}
    \right)+\left(
      \begin{array}{cc}
        0 & yu-xv \\
        0 & 0 \\
      \end{array}
    \right)\\
 &=&\left(
      \begin{array}{cc}
        xy-yx & 0\\
        0 & 0 \\
      \end{array}
    \right).
\end{eqnarray*}
Therefore, the descendent Lie algebra $(\g ,[\cdot,\cdot]_{\Diff(\huaB)})$ is the direct sum of the Lie algebra   $\so(n)$ and the abelian Lie algebra $\R^{n}$, and obviously $\so(n)$ is an ideal. By Proposition \ref{pro:rep-diff}, the representation
  $\theta: \g\lon \gl(\g)$ of the descendent Lie algebra $(\g ,[\cdot,\cdot]_{\Diff(\huaB)})$ on $\g$ is given by
$$
\theta(\left(
       \begin{array}{cc}
         x & u \\
         0 & 0 \\
       \end{array}
     \right)
)\left(
   \begin{array}{cc}
     y & v \\
     0 & 0 \\
   \end{array}
 \right)=\frac{d}{dt}\bigg|_{t=0}\Theta(\left(
                                     \begin{array}{cc}
                                       \Exp(tx) & tu \\
                                       0 & 1 \\
                                     \end{array}
                                   \right))\left(
                                             \begin{array}{cc}
                                               y & v \\
                                               0 & 0 \\
                                             \end{array}
                                           \right)
                                   =\left(
                                      \begin{array}{cc}
                                        0 & xv \\
                                        0 & 0 \\
                                      \end{array}
                                    \right).
$$

Denote by $\g^{\so(n)}=\{\xi\in\g|\theta(\left(
                                          \begin{array}{cc}
                                            x & 0 \\
                                            0 & 0 \\
                                          \end{array}
                                        \right)
)\xi=0, \forall x\in\so(n)\}.
$
Then it is obvious that $\g^{\so(n)}=\{\left(
                               \begin{array}{cc}
                                 \so(n) & 0 \\
                                 0 & 0 \\
                               \end{array}
                             \right)\}.$
With above preparations, using the exact sequences of low degree terms in the Hochschild-Serre spectral sequence, we obtain
\begin{equation*}
0\lon H^{1}(\g/{\so(n)}, \g^{{\so(n)}})\lon H^{1}(\g,\g) \lon H^{1}({\so(n)},\g)^{\g/{\so(n)}}\lon H^{2}(\g/{\so(n)}, \g^{{\so(n)}})\lon H^{2}(\g,\g).
\end{equation*}
Since $\so(n)$ is a semisimple Lie algebra, we have $H^{1}({\so(n)},\g)^{\g/{\so(n)}}=0$, thus $H^{1}(\g/{\so(n)},\g^{{\so(n)}})\cong H^{1}(\g,\g)$. Since $\g/{\so(n)}=\mathbb R^n$ is abelian and the representation is trivial, it follows that $$H^{1}(\g/{\so(n)},\g^{{\so(n)}})=\Hom(\mathbb R^n,{\so(n)})\cong\R^{\frac{n^{2}(n-1)}{2}}.$$ Therefore, we have $$H^{2}(\Diff(\huaB))=H^{1}(\g,\g)\cong\R^{\frac{n^{2}(n-1)}{2}}.$$
}
\end{ex}

\section{Local integration of relative Rota-Baxter operators on Lie algebras}\label{sec:loi}

In this section, we introduce the notion of a local relative Rota-Baxter operator on a Lie group and show that any relative Rota-Baxter operator on a Lie algebra can be integrated to a   local relative Rota-Baxter operator on the Lie group integrating the Lie algebra $\g.$

\emptycomment{1. integrate $B:(\h, [-,-]_B) \to \g$ to $\huaB: \tilde{H} \to G$,
where both $\tilde{H}$ and $G$ are connected and simply connected.

2. prove that $h_1 \cdot_{\tilde{H}} h_2 = h_1 \cdot_H
\Phi(\huaB(h_1))h_2$ holds locally. Sketch of proof: $Lie(\huaH)=(\h, [-,-]_B)\cong
Graph_1(B) \subset \g\rtimes_\phi \h $, and $\huaH \cong
Graph(\huaB) (\subset G\times \tilde{H})$
whose Lie algebra $Graph_2(B) \subset \g \times (\h, [-, -]_B)$ is
isomorphic to $Graph_1(B) \subset \g\rtimes_\phi \h$. Let us call $E$ the
sub Lie group corresponding to $Graph_1(B)$.
Then $E$ is locally isomorphic to $Graph(\huaB)$ as manifold near the
identity. Moreover $U(E)$ and $U(\tilde{H})$ are isomorphic as local Lie
groups.

3. prove that $\huaB(h):=\exp^G(B(\log^{\tilde{H}} h))$ provides a
local integration based on observation in 2. Here $\log$ is the
inverse function of $\exp$ near identity.

4. prove $\Hom(l.int(B_1), \huaB_2) = \Hom(B_1, diff(\huaB_2))$. Here
$l.int$ is the functor from R.B.O. on LieAlg to that on local LieGp,
and $diff$ is the functor goes the other way around.
}

\begin{defi}
Let $\Phi: G\lon \Aut(H)$ be an action of $G$ on $H$. If there exists an open neighborhood $U$ of $e_H$ in $H$ and a smooth map $\huaB: U\lon G$ such that
\begin{equation}\label{eq:loc-RB}
\huaB(h_1)\cdot_{G}\huaB(h_2)=\huaB(h_1\cdot_H\Phi(\huaB(h_1))h_2), \quad h_1, h_2\in U,
\end{equation}
whenever the element $h_1\cdot_H\Phi(\huaB(h_1))h_2$ is also in $U$,
then $\huaB$ is called a {\bf local relative Rota-Baxter operator} on $G$
with respect to the action $(H;\Phi)$.

In particular, if $H=G$ and the action is the adjoint action of $G$ on itself, then we call $\huaB$ a {\bf local Rota-Baxter operator}.

Two local relative Rota-Baxter operators $\huaB: U\lon G$ and $\huaB':
U'\lon G$ are defined to be equal if there exists an open neighborhood $\bar{U}\subset U\cap U'$ of $e_H$, such that $B|_{\bar{U}}=B'|_{\bar{U}}.$
\end{defi}

\begin{defi}
 Let  $\huaB: U\lon G$ and $\huaB': U'\lon G$ be two local relative Rota-Baxter operators on $G$ with respect to an action $(H;\Phi)$.  A {\bf homomorphism} from $\huaB'$ to $\huaB$ consists of  Lie group homomorphisms $\Psi_G: G\lon G$ and   $\Psi_H: H\lon H$ such that \eqref{hom-rbo-gp1} holds in an open neighborhood $\bar{U}\subset U\cap U'$ of $e_H$ and \eqref{hom-rbo-gp2} holds.
\end{defi}

\subsection{Local integration functor and adjointness}

It is clear that local relative Rota-Baxter operators on a Lie group
$G$ with respect to an action $(H;\Phi)$ together with
homomorphisms between them also form a category, which we denote by
$\lRBGH$. 
Since the differentiation functor established in the previous section
depends only on the local information near the identity element, it induces a
differentiation functor,
\begin{equation}\label{eq:diff-functor}\Diff: \lRBGH
\to \RBgh, \end{equation} which we use the same letter by abusing the
notations.

\emptycomment{
Similar as the discussion in previous sections, given a local relative Rota-Baxter operator $\huaB$, taking the differentiation, we can obtain a relative Rota-Baxter operator $$B=\Diff(\huaB):=\huaB_*:\h\lon\g.$$ Moreover, given a homomorphism  $(\Psi_G,\Psi_H)$, taking the differentiation, we can obtain a homomorphism $(\psi_G,\psi_H)$. Consequently, there is a functor
$$
\Diff:\RBGH\to \RBgh.
$$ }

Let  $B:\h\lon\g$ be a relative Rota-Baxter operator on
$\g$ with respect to an action $(\h;\phi)$. Let $G$ and $H$ be
  Lie
groups (not necessarily connected and simply connected) integrating $\g$ and $\h$ respectively. If there is an action
$\Phi: G\to \Aut(H)$ whose differentiation is $\phi$ (see \eqref{eq:relation}), and a local relative
Rota-Baxter operator $\huaB$ on $G$ with respect to $\Phi$ such that $\Diff(\huaB)=B$, then we call that $\huaB$
{\bf integrates} $B$. 

\emptycomment{
Let $\phi: \g\lon \Der(\h)$ be a Lie algebra homomorphism. Let  $G$
and $H$ be  connected, simply connected Lie groups integrating Lie
algebras $\g$ and $\h$ respectively.  First there is a unique Lie
group homomorphism $\Phi: G\lon \Aut(\h)$ such that
$\Phi_{*}=\phi$. Moreover, since $H$ is a connected, simply connected
Lie group, it follows that $\Aut(\h)\cong\Aut(H)$. Therefore  we have
a Lie group homomorphism from $G$ to $\Aut(H)$, which is also denoted
by $\Phi: G\lon\Aut(H)$. The relation between $\Phi$ and $\phi$ is
illustrated by Diagram \eqref{eq:relation}. }

\begin{thm} \label{thm:obj}
Let $B: \h\lon \g$ be a relative Rota-Baxter
operator on a Lie algebra $\g$ with respect to an action $(\h;\phi)$. Let $G$ and $H$ be the connected and simply connected Lie
groups integrating $\g$ and $\h$ respectively, and $\Phi: G\to \Aut(H)$ be the
integrated action (see the paragraph above \eqref{eq:relation}).
Then there is a unique local relative Rota-Baxter operator $\huaB$ on
$G$ with respect to the action $(H;\Phi)$
integrating $B$.
\end{thm}

\begin{proof}
We consider the semi-direct product Lie group $G\ltimes_\Phi H$, with
multiplication $\cdot_\Phi$ given by
\begin{equation*}
(g_1,h_1)\cdot_{\Phi}(g_2,h_2)=(g_1\cdot_Gg_2, h_1\cdot_H\Phi(g_1)h_2), \quad \forall g_i\in G, h_i\in H, i=1,2.
\end{equation*}
Its Lie algebra is the semi-direct product Lie algebra $\g\ltimes_\phi
\h$, and by Proposition \ref{graphro},   $Gr(B)$  is a Lie subalgebra
of $\g\ltimes_\phi \h$. Thus there exists a connected Lie subgroup $E$ of
$G\ltimes_\Phi H$ such that its Lie algebra is $Gr(B)$.

Define a smooth map $P_H: E\lon H$ by
$
P_H(g,h)=h $ for all $(g,h)\in E.
$
Then its tangent map at the identity ${P_H}_*: Gr(B)\lon\h$ is given by
\begin{equation*}
{P_H}_{*}(B(u),u)=u,\quad \forall u\in\h,
\end{equation*}
which is an isomorphism from the vector space $Gr(B)$ to the vector
space $\h$.  Thus there exists an open set $V\subseteq E$ such that
$P_H|_{V}$ is an isomorphism of manifolds, which implies that there is
an open set $U$ of $H$ which contains $e_H$ and a smooth map $\huaB: U\lon G$, such that $V\cong Gr(\huaB)$.

For all $h_1,h_2\in U$, we have
\begin{equation*}
(\huaB(h_1),h_1)\cdot_\Phi (\huaB(h_2),h_2)=(\huaB(h_1)\cdot_G\huaB(h_2), h_1\cdot_H\Phi(\huaB(h_1))h_2).
\end{equation*}
Therefore, as soon as $h_1\cdot_H\Phi(\huaB(h_1))h_2\in U$, we have
$$(\huaB(h_1)\cdot_G\huaB(h_2), h_1\cdot_H\Phi(\huaB(h_1))h_2)\in Gr({\huaB}),$$
which implies that
$\huaB(h_1)\cdot_G\huaB(h_2)=\huaB(h_1\cdot_H\Phi(\huaB(h_1))h_2)$. Therefore,
$\huaB:U\lon G$ is a local relative Rota-Baxter
operator. Furthermore, since $Gr(\huaB)\subset E$, and the Lie algebra
of $E$ is $Gr(B)$, it follows that $\huaB_*=B$, and $\huaB$ is an
integration of $B$.

Now we show the uniqueness of $\huaB$ through an explicit
formula. Denote by $\EXP$ the exponential map for the Lie group
$G\ltimes_\Phi H$, and by $P_H$ the projection $G\ltimes_{\Phi}
H\to H$. For all $x\in\g, u\in\h$, it is obvious that
\begin{equation*}
\EXP (x,u)=(\exp_Gx, P_H\EXP(x,u)).
\end{equation*}
Since the Lie algebra of the Lie subgroup $E$  is $Gr(B)$, it follows that locally $\EXP(B(u),u)\in Gr(\huaB)\subseteq E$. Therefore,
\begin{equation} \label{eq:exp}
\huaB(P_H(\EXP (B(u),u)))=\exp_{G}B(u).
\end{equation}
Since $P_H\circ \EXP: Gr(B) \to H$ is
a local isomorphism in the neighborhood of the identity\footnote{The
  map itself is also defined locally therein.}, the uniqueness
of $\huaB$ naturally follows.
\end{proof}

Now we extend the integration to the level of morphisms and establish
a functor $$\Int: \RBgh\to \lRBGH .$$

\begin{thm} \label{thm:int-functor} Let $B$ and $B'$ be relative Rota-Baxter operators on a
  Lie algebra $\g$ with respect to an action $(\h;\phi)$
   and $\psi=(\psi_\g, \psi_\h)$ be a
  homomorphism from $B'$ to $B$. Let $G$ and $H$ be connected and
  simply connected Lie groups integrating $\g$ and $\h$ respectively,
  and let $\huaB': U'\to G$ and $\huaB:
  U\to G$ be the integrated local relative Rota-Baxter operators of $B'$
  and $B$ respectively,  as in
  the previous theorem.
Let $\Psi_G$ and $\Psi_H$ be the Lie group homomorphisms integrating the
Lie algebra homomorphisms $\psi_\g$ and $\psi_\h$ respectively.  Then
$(\Psi_G,\Psi_H)$ is a homomorphism from $\huaB'$ to
$\huaB$. Consequently, we obtain a functor
\begin{equation}\label{eq:int-functor}
\Int:\RBgh\to \lRBGH.
\end{equation}
\end{thm}
\begin{proof}
To prove that $(\Psi_G,\Psi_H)$ is a
 homomorphism from $\huaB'$ to $\huaB$, we need to show
 Eq. \eqref{hom-rbo-gp2}, and a local version of Eq. \eqref{hom-rbo-gp1}. Eq. \eqref{hom-rbo-gp2} does not
 involve relative Rota-Baxter operators and follows directly from
 Lemma \ref{pro-imporpair}.

Now to show Eq. \eqref{hom-rbo-gp1} holds locally, we need to show
that
 $\huaB\circ\Psi_H=\Psi_G\circ\huaB'$ holds in an open neighborhood
 $\bar{U}\subset U\cap U'$ of $e_H$. This follows from the explicit
 formula \eqref{eq:exp},  the corresponding infinitesimal condition
 \eqref{hom-rbo1}, and the fact that $(\Psi_G, \Psi_H): G\ltimes_\Phi
 H \to G\ltimes_\Phi H$ is the integrated homomorphism of the Lie algebra homomorphism $(\psi_\g,
 \psi_\h): \g\ltimes_\phi \h \to \g \ltimes_\phi \h$ (i.e. Lemma \ref{pro-imporpair}). More precisely, we have
  \begin{eqnarray*}
   \Psi_{G}\Big(\huaB'(P_H(\EXP (B'(u),u)))\Big)&=&\Psi_{G}(\exp_{G}B'(u))
=\exp_{G}(\psi_{\g}B'(u))
= \exp_{G}(B(\psi_{\h}(u)))\\
&=&\huaB(P_H(\EXP (B(\psi_{\h}(u)), \psi_{\h}(u))))=\huaB(P_H(\EXP (\psi_{\g}(B'(u)), \psi_{\h}(u))))
\\&=&\huaB(P_H((\Psi_G,\Psi_H)\EXP (B'(u),u)))=\huaB\Psi_H(P_H(\EXP (B'(u),u))).
\end{eqnarray*}
Thus we obtain
\begin{equation*}
\Psi_G\circ \huaB'=\huaB\circ \Psi_H,
\end{equation*}
which implies that $(\Psi_G,\Psi_H)$ is a homomorphism from $\huaB'$ to $\huaB$.  Then it is straightforward to see that $\Int$ is a functor.
\end{proof}

\begin{thm}\label{thm:adj}
The integration functor $\Int$ in \eqref{eq:int-functor} and the
differentiation functor $\Diff$ in \eqref{eq:diff-functor} are adjoint
functors. More precisely,   $\Int$ is left adjoint to $\Diff$ and
$\Diff$ is right adjoint to $\Int$.
\end{thm}
\begin{proof}
We need to show that, there is an isomorphism $\alpha_{B'
  \huaB}$, such that
 \begin{equation}\label{eq:alpha}
   \alpha_{B'\huaB}: \Hom_{\lRBGH}(\Int(B'),\huaB)
   \xrightarrow{\cong}\Hom_{\RBgh}(B',\Diff(\huaB)), \quad \forall B'\in
   \RBgh,  \huaB \in \lRBGH,
\end{equation} and $\alpha$ is bi-natural in $B'$ and in $\huaB$, that
is,  when fixing $\huaB$, $\alpha_{-\huaB}$ is a natural isomorphism
between functors $\Hom_{\lRBGH}(\Int(-),\huaB)$ and
$\Hom_{\RBgh}(-,\Diff(\huaB))$; and when fixing $B'$,
$\alpha_{B'-}$ is a natural isomorphism
$\Hom_{\lRBGH}(\Int(B'), -)
   \xrightarrow{\cong}\Hom_{\RBgh}(B',\Diff(-)) $.

We see that by the definition of $\Int$ in Theorem \ref{thm:int-functor},
$G$ and $H$ are required to be connected and simply connected. Given a
morphism $\Psi :=(\Psi_G,
\Psi_H)  \in \Hom_{\lRBGH}(\Int(B'),\huaB)$, it is not hard to
verify that $\alpha_{B'\huaB}: \Psi\mapsto \Diff(\Psi)$,  as
constructed in Proposition \ref{pro:diff-functor},
is an isomorphism with the inverse constructed in Theorem
\ref{thm:int-functor}. The bi-naturality of $\alpha$ follows
from the fact that $\alpha_{B'\huaB}$ preserves the composition, which in
turn follows from the functoriality of $\Diff$. We demonstrate in the
following the naturality in $B'$ and leave the other to interested readers.

Let $f': B'_1\to B'_2$ be a morphism in $\RBgh$. The naturality of
$\alpha_{-\huaB}$ follows from
the commutativity of the diagram,
\begin{equation}\label{diag:alpha-huaB}
\xymatrix{
\Hom_{\lRBGH}(\Int(B'_1),\huaB)
   \ar[r]^{\alpha_{B'_1\huaB}} &\Hom_{\RBgh}(B'_1,\Diff(\huaB)) \\
\Hom_{\lRBGH}(\Int(B'_2),\huaB)
   \ar[r]^{\alpha_{B'_2\huaB}}
   \ar[u]^{\Int(f')^*} & \Hom_{\RBgh} (B'_2,\Diff(\huaB)) \ar[u]_{f'^*},
}
\end{equation} where $\;^*$ denotes the
precomposition (or pullback).  In fact, for all $\Psi\in \Hom_{\lRBGH}(\Int(B'_2),\huaB)$, we have
\begin{equation}
\alpha_{B'_1\huaB} \Int(f')^*(\Psi)=\alpha_{B'_1\huaB}(\Psi\circ \Int(f')) =
\Diff(\Psi\circ \Int(f'))=\Diff(\Psi) \circ f' = f'^*
\alpha_{B'_2\huaB} (\Psi),
\end{equation}
\emptycomment{
Let $f: B''\to B'$ be a morphism in $\RBgh$, then \comment{The meaning of two $f_*$ is not clear}
\begin{equation}
\alpha_{B''\huaB} f_*(\Psi)=\alpha_{B''\huaB}(\Psi\circ \Int(f)) =
\Diff(\Psi\circ \Int(f))=\Diff(\Psi) \circ f = f_* \circ
\alpha_{B'\huaB} (\Psi),
\end{equation}
}
which implies that $\alpha_{-\huaB}$ is natural.
\end{proof}

 \section{Geometric applications} \label{sec:app}
 In this section, we give some geometric applications including the explicit expression of the local factorization of a Lie group given in \cite{STS} and the integration of matched pairs of Lie algebras given by Rota-Baxter operators.

 \subsection{Local descendent Lie group}
As in the global case, a local Rota-Baxter operator also
determines a local descendent Lie group. This turns out to be important to our geometric application. Thus we make this construction here in this subsection.

While the definition of a Lie
group is standard, the definition of a local Lie group varies
slightly. Thus we first fix our definition which is taken mostly from the classical one in \cite{pontryagin}, and however with a slight modification similar to the more recent paper \cite{Olver:96}.

\begin{defi} {\rm(\cite{pontryagin,Olver:96})}
 A {\bf local Lie group} $L$ is a manifold equipped with an open set $V\subset L$ together with
 \begin{enumerate}
     \item[{\rm(i)}] a multiplication map $m: V\times V \to L$, which satisfies $g_1(g_2g_3)=(g_1g_2)g_3$, as long as $g_1, g_2, g_3, g_1g_2, g_2g_3 \in V$ (here we write $m(g_1, g_2)$ as $g_1g_2$ for short),
     \item[{\rm(ii)}] an identity element $e\in V$, which satisfies $eg=ge=g$, for all  $  g\in V$,
     \item[{\rm(iii)}] an inverse map $i: V\to V$, such that $i(V)=V$, and $g^{-1}g=e=gg^{-1}$ (here we write $i(g)$ as $g^{-1}$ for short).
 \end{enumerate}
\end{defi}

\begin{thm}\label{thm:local-des}
Let $\huaB: U\to G$ be a
local  Rota-Baxter operator on the Lie group $(G,e,\cdot)$. Then it gives  rise to a local Lie group structure on $G:$
\begin{itemize}
  \item[$\bullet$]the   multiplication $\star$ is given by
\begin{equation}\label{eq:mul-loc}
    g_1 \star g_2 := g_1 \cdot\huaB (g_1) \cdot g_2\cdot \huaB(g_1)^{-1},
\end{equation}
\item[$\bullet$]the identity element is $e$, which is the identity element of the Lie group $G$,

\item[$\bullet$] the inverse $g^\dag$ of $g$ is given by
\begin{equation}\label{eq:inv-loc}
    g^\dag := \huaB(g)^{-1} \cdot g^{-1}\cdot \huaB(g).
\end{equation}
\end{itemize}

\end{thm}
\begin{proof}
Notice that the inverse is defined from $U$ to $G$, thus we may take $V:=U\cap U^\dag$. Clearly $V$ is an open set of $G$ containing $e$. Note that $\star: U\times U \to G$ is well-defined, thus $\star$ is well-defined also on $V\times V$.

If $g_1, g_2, g_3, g_1\star g_2, g_2\star g_3 \in V$, then the same calculation \eqref{eq:proof-ass} for proving the associativity in Proposition \ref{pro:dec-gp} still goes through because all the relevant elements are in $V$ which is within the definition domain of $\huaB$.

By \eqref{eq:loc-RB}, for all $g\in U$, we have $\huaB(g)\cdot\huaB(g^\dag ) = \huaB(g \cdot\huaB(g)\cdot g^\dag \cdot\huaB(g)^{-1})=\huaB(e)=e$.  This shows that $\huaB(g^\dag)=\huaB(g)^{-1}$. Thus
\begin{equation}
    (g^\dag)^\dag = \huaB(g)\cdot (g^\dag)^{-1} \cdot\huaB(g)^{-1} = g.
\end{equation} Therefore, $(U^\dag)^\dag = U$, which implies that $(V)^\dag = V$. Moreover, for all $ g\in V$, $g^\dag \star g = e = g\star g^\dag$ follows from the same calculation \eqref{eq:proof-inv}   in Proposition \ref{pro:dec-gp} because $\huaB(g^\dag)=\huaB(g)^{-1}$ still holds for $g \in U$. Thus $(V\subset G, \star, \dag)$ is a local Lie group.
\end{proof}
This local Lie group is called the {\bf descendent local Lie group} of $\huaB$.

\begin{rmk}\label{rmk:local-des-diff}
Similar to the case of the descendent Lie group, it is clear that
\begin{itemize}
    \item  the Lie algebra of the descendent local Lie group $(V\subset G, \star, \dag)$ is  the descendent Lie algebra $(\g, [\cdot,\cdot]_B)$, where $B=\Diff(\huaB)$.
    \item $\huaB: (V\subset G, \star, \dag) \to G$ is a local Lie group homomorphism. This in particular implies that $\Img (\huaB) $ and $G_-$ are locally isomorphic near the identity, where $G_-$ is the Lie group integrating the Lie algebra  $\g_-:=\Img (B)$.
\end{itemize}
\end{rmk}

\subsection{Application in the local factorization problem}
As mentioned in Remark \ref{rmk:mybe},   Rota-Baxter operators $B$ on a Lie algebra $\g$ one-to-one correspond  to   modified $r$-matrices $R$. The study of integration of such modified $r$-matrices gives arise to the Adler-Kostant-Symes (AKS) theory \cite{STS}. The AKS theory \cite{Camille-book} allows one to construct an explicit integral curve $L(t)$ of certain Hamiltonian function $h$ on $\g^*$, by
\begin{equation}
 L(t)=\Ad^*_{ g_{+}(t)} L_0=\Ad^*_{g_{-}(t)} L_0,
\end{equation}
with initial value $L_0 \in \g^*$, as long as we know the solution $(g_{-}(t), g_{+}(t))$ of the local factorization problem, which in turn comes from a split (i.e. an infinitesimal factorization) of $\g$. However, usually, the solution $(g_{-}(t), g_{+}(t))$ is not explicitly given. With the help of the local integration of Rota-Baxter operators, we can give such a local factorization explicitly.

To state this more precisely, let us first recall some facts adapted to the Rota-Baxter operators context from   \cite{GLS,STS}, and we refer the readers to the references therein for the original references.  Let $B$ be a Rota-Baxter operator on $\g$. Then it gives rise to an infinitesimal factorization (or a split) \cite[Proposition 9]{STS} of $\g$ as follows. Denote by $ \g_{+}=\Img(B+\Id), ~\g_{-}=\Img (B).$
 We further denote by $\mathfrak{k}_+:= \ker (B)$ and $\mathfrak{k}_-:=\ker (B+\Id)$. Since both $B$ and $B+\Id$ are Lie algebra homomorphisms from the descendent Lie algebra $\g_B$ to $\g$, $\g_+,\g_-,\mathfrak{k}_+$ and $\mathfrak{k}_-$ are Lie subalgebras. Moreover,  $\mathfrak{k}_{\pm}\subset \g_{\pm}$ are Lie ideals, and $\theta: \g_+/\mathfrak{k}_+ \to \g_-/\mathfrak{k}_-$ given by $(B+\Id)(u) \mapsto B(u)$ on the representatives of equivalence classes is well-defined and is a Lie algebra isomorphism. The map $\theta$ is called the {\em Cayley transform} of $B$. Define $    \g_\theta = \g_{+} \oplus \g_{-} $ by
   $$
   \g_\theta=\{(x,y)\in\g_{+} \oplus \g_{-}~\mbox{ such that }~\theta (\bar{x})=\bar{y} \}.
   $$
   Then any $x\in\g$ can be uniquely expressed as $x=x_+-x_-$ for $(x_+,x_-)\in\g_\theta.$

   Let $G$ be a Lie group of $\g$, and let $G_{\pm}$ and $K_{\pm}$ be the Lie subgroups of $G$ integrating the above Lie subalgebras $\g_{\pm}$ and $\mathfrak{k}_{\pm}$ correspondingly. Then the Cayley transform $\theta$ integrates also to the group level into a (local) Lie group homomorphism $\Theta: G_{+}/K_+\to G_-/K_-$, by Lie's II Theorem. Let us denote by $\bar{g}$ the equivalence class of $g$. Then the solution $( g_{+}(t),g_{-}(t)) \in G_+\times G_-$ of the (local) factorization problem in \cite[Theorem 11]{STS} is the solution of the following equations,
\begin{equation*}
    \exp 2t X_0 = g_{+}(t)\cdot g_{-}(t)^{-1}, \quad \Theta(\bar{g}_+(t)) = \bar{g}_{-}(t), \quad \text{where} \: X_0 = dh|_{L_0} \in \g,
\end{equation*} with the initial value $g_{\pm}(0)=e\in G$, for sufficiently small $t$.

\begin{thm}
Let $(\g, [\cdot, \cdot]_\g, B)$ be a Rota-Baxter Lie algebra.  Then the above solution $(g_{+},g_{-})$ of the (local) factorization problem
has an explicit expression,
\begin{equation}\label{eq:solu}
    g_- (t) = \huaB ( \exp 2t X_0), \quad g_+(t) = \exp 2tX_0 \cdot\huaB (\exp 2tX_0),
\end{equation}
for sufficiently small $t$, with $\huaB $ given explicitly by \eqref{eq:exp} locally.
\end{thm}
\begin{proof}
Let   $(G, \cdot, \huaB)$ be the integrated local Rota-Baxter Lie group of $(\g, [\cdot, \cdot]_\g, B)$ given in Theorem \ref{thm:obj}. By Remark \ref{rmk:local-des-diff},   $\Img(\huaB)$ and    $G_-$ are locally isomorphic. Define  $\huaB_+:U\to G$   by $$\huaB_+(g)=g\cdot\huaB(g).$$
By the same discussion in \cite[Proposition 3.1]{GLS},  $\huaB_+$ is a local Lie group homomorphism from the descendent local Lie group to  $G$. Since $\Img(\huaB_+)$ and $G_+$ have the same Lie algebra $\g_+$, so they are also locally isomorphic. Obviously, locally we have
$$
g=g\cdot\huaB(g)\cdot(\huaB(g))^{-1},\quad g\cdot\huaB(g)\in\Img(\huaB_+),~ \huaB(g)\in \Img(\huaB).
$$
Therefore, we have $\exp 2t X_0 = g_{+}(t)\cdot g_{-}(t)^{-1}$, where $g_{+}(t)$ and $ g_{-}(t)$ are given by \eqref{eq:solu}.
\end{proof}

\subsection{Application in the integration of matched pairs}
A {\bf matched pair of Lie algebras}  consists of a pair of Lie algebras  $(\g,\h)$, a  representation $\rho: \g\to\gl(\h)$ of $\g$ on $\h$ and a   representation $\mu: \h\to\gl(\g)$ of $\h$ on $\g$ such that
\begin{eqnarray}
\label{eq:mp1}\rho(x) [\xi,\eta]_{\h}&=&[\rho(x)\xi,\eta]_{\h}+[\xi,\rho(x) \eta]_{\h}+\rho\big((\mu(\eta)x\big)\xi-\rho\big(\mu(\xi)x\big) \eta,\\
\label{eq:mp2}\mu(\xi) [x, y]_{\g}&=&[\mu(\xi)x,y]_{\g}+[x,\mu(\xi) y]_{\g}+\mu\big(\rho(y) \xi\big)x-\mu\big(\rho(x)\xi\big)y,
\end{eqnarray}
 for all $x,y\in \g$ and $\xi,\eta\in \h$. 

It is well known that if there is a Lie algebra $\frkk$ such that both $\g$ and $\h$ are Lie subalgebras, and $\frkk$ is isomorphic to $\g\oplus \h$ as vector spaces, then $(\g,\h)$ is a matched pair of Lie algebras.
\begin{lem}
Let $(\g, [\cdot, \cdot]_\g, B)$ be a Rota-Baxter Lie algebra. Then $(\g_{B}, \g_{diag})$ is a matched pair of Lie algebras, where $\g_{diag}=\{(x, x)|\forall x\in\g\}$ and $\g_B=\{(B(x), x+B(x))|\forall x\in\g\}$.
\end{lem}
\begin{proof}
Consider  the  direct sum Lie algebra  $\g\oplus\g$, in which the Lie bracket is given by
$$
[(x,y),(x',y')]=([x,x']_\g,[y,y']_\g).
$$
It is obvious that $\g_{diag}$ is a Lie subalgebra of $\g\oplus \g$.  Since $(\g, [\cdot, \cdot]_\g, B)$ is a Rota-Baxter Lie algebra, we have
\begin{eqnarray*}
&&[(B(x), x+B(x)), (B(y), y+B(y))]\\
&=&([B(x), B(y)]_{\g}, [x, y]_\g+[x, B(y)]_\g+[B(x), y]_\g+[B(x), B(y)]_\g)\\
&=&\Big(B([x, y]_\g+[x, B(y)]_\g+[B(x), y]_\g), \\
&&[x, y]_\g+[x, B(y)]_\g+[B(x), y]_\g+B([x, y]_\g+[x, B(y)]_\g+[B(x), y]_\g)\Big)\\
&\in&\g_B,
\end{eqnarray*}
which implies that  $\g_B$ is a Lie subalgebra of $\g\oplus\g$. Moreover, it is obvious that $\g_{diag}\oplus\g_B$ is isomorphic to $\g\oplus\g$ as vector spaces. Therefore, $(\g_{B}, \g_{diag})$ is a matched pair of Lie algebras.
\end{proof}

A pair of Lie groups $(P,Q)$ is called a {\bf matched pair of Lie groups} (\cite{Ma}) if there is a left action of $P$ on $Q$ and a right action of $Q$ on $P$:
\[P\times Q\to Q,\qquad (p,q)\mapsto  p\triangleright q;\qquad P\times Q\to P\qquad (p,q)\mapsto p\triangleleft q,\]
such that
\begin{eqnarray}
\label{mpg1} p\triangleright (q_1q_2)&=&(p\triangleright q_1)\big((p\triangleleft q_1)\triangleright q_2\big);\\
\label{mpg2} (p_1p_2)\triangleleft q&=&\big(p_1\triangleleft (p_2\triangleright q)\big)   (p_2\triangleleft q).
\end{eqnarray}

Similar to the case of Lie algebras, the following equivalent characterization of matched pairs of Lie groups is well known.

\begin{pro}\label{pro:mpg}
 A pair of Lie groups $(P,Q)$ is a  matched pair if there exists a Lie group $K$, and injective Lie group homomorphism $i_P:P\to K$ and  $i_Q:Q\to K$, such that  the map $P\times Q\lon K$ defined by $(p, q)\to i_P(p)i_Q(q)$ is a diffeomorphism.
\end{pro}
It is not hard to see that a matched pair of Lie groups differentiates to a matched pair of Lie algebras. However, not every matched pair of Lie algebras integrates to a matched pair of Lie groups (e.g.
\cite[Example 2.7]{ES} gives rise to a counter example). Some partial result is available for this integration problem. For example, it is proved in \cite{Ma} that a matched pair of Lie algebras $(\g, \h)$ integrates to a matched pair of Lie groups $(G, H)$, if the simply connected Lie groups $G$ and $H$ integrating $\g$ and $\h$ respectively are compact. In the following, we give another integration result for matched pairs coming from Rota-Baxter operators using the integration of Rota-Baxter operators.

\begin{thm}
Let $(\g, [\cdot, \cdot]_\g, B)$ be a Rota-Baxter Lie algebra. Assume that $B$ is integrable and $(G, \cdot, \huaB)$ is the integrated Rota-Baxter Lie group of $(\g, [\cdot, \cdot]_\g, B)$. Then $( G_{\huaB}, G_{diag})$ is a matched pair of Lie groups, where $G_{diag}=\{(g, g)|\forall g\in G\}$ and $G_{\huaB}=\{(\huaB(g), g\cdot\huaB(g))|\forall g\in G\}$. Furthermore, the differentiation of  $( G_{\huaB}, G_{diag})$ is the matched pair $(\g_{B}, \g_{diag})$.
\end{thm}
\begin{proof}
Consider the  direct product Lie group $G\times G$, in which the group multiplication is given by
$$
(g,h)(g',h')=(g\cdot g',h\cdot h').
$$
It is obvious that $G_{diag}$ is a Lie subgroup.  Since $(G, \cdot, \huaB)$ is a Rota-Baxter Lie group, for any $g, h\in G$, we have
\begin{eqnarray*}
&&\Big(\huaB(g), g\cdot\huaB(g)\Big)\Big(\huaB(h), h\cdot\huaB(h)\Big)\\
&=&\Big(\huaB(g)\cdot\huaB(h), g\cdot\huaB(g)\cdot h\cdot\huaB(h)\Big)\\
&=&\Big(\huaB(g\cdot\huaB(g)\cdot h\cdot(\huaB(g))^{-1}), g\cdot\huaB(g)\cdot h\cdot(\huaB(g))^{-1}\cdot\huaB(g\cdot\huaB(g)\cdot h\cdot(\huaB(g))^{-1})\Big)\\
&\in&G_\huaB
\end{eqnarray*}
and
\begin{eqnarray*}
(\huaB(g), g\cdot\huaB(g))^{-1}&=&\Big(\huaB((\huaB(g))^{-1}\cdot g^{-1}\cdot\huaB(g)), (\huaB(g))^{-1}\cdot g^{-1}\cdot\huaB(g)\cdot \huaB((\huaB(g))^{-1}\cdot g^{-1}\cdot\huaB(g))\Big)\\
&\in&G_\huaB,
\end{eqnarray*}
which implies that $G_{\huaB}$ is a subgroup. Moreover, since  $\huaB$ is a smooth map, it follows that $G_\huaB$ is a closed set in $G\times G$. Thus $G_{\huaB}$ is a Lie subgroup.
For any $(a, b)\in G\times G$, there is
$$
(a, b)=\Big(\huaB(b\cdot a^{-1}), b\cdot a^{-1}\cdot\huaB(b\cdot a^{-1})\Big)\Big((\huaB(b\cdot a^{-1}))^{-1}\cdot a, (\huaB(b\cdot a^{-1}))^{-1}\cdot a\Big).
$$
Thus the map $G_\huaB\times G_{diag}\lon G\times G$ defined by $\Big((\huaB(g), g\cdot\huaB(g)), (h, h)\Big)\lon \Big(\huaB(g)\cdot h, g\cdot\huaB(g)\cdot h\Big)$ is a diffeomorphism. Therefore, by Proposition \ref{pro:mpg}, $(G_{\huaB}, G_{diag})$ is a matched pair of Lie groups. The other conclusion is straightforward.
\end{proof}

\begin{rmk}
Despite there are explicit counter examples of nonintegrable matched pairs, we do not know yet an explicit example of nonintegrable Rota-Baxter operators.
\end{rmk}

 \end{document}